\documentclass[12pt]{amsart}
\usepackage{geometry}                
\usepackage{graphicx}
\usepackage{amssymb, enumerate}
\usepackage[matrix,arrow,ps]{xy}
\usepackage{epstopdf}
\newcommand{\real}{{\rm Re}\,}
\newcommand{\imag}{{\rm Im}\,}
\newcommand{\C}{\mathbb C}
\newcommand{\R}{\mathbb R}

\newtheorem{theorem}{Theorem}[section]
\newtheorem{lemma}[theorem]{Lemma}

\newtheorem{prop}[theorem]{Proposition}

\theoremstyle{remark}
\newtheorem{remark}[theorem]{Remark}
\theoremstyle{example}
\newtheorem{example}[theorem]{Example}
\theoremstyle{definition}
\newtheorem{definition}[theorem]{Definition}
\numberwithin{equation}{section}

\begin{document}

\begin{abstract}
We construct a finitely dimensional invariant manifold of   
 holomorphic discs attached to a certain class of smooth pseudconvex hypersurfaces 
 of finite type in $\C^2$, generalizing the notion of stationary discs. The discs we construct 
 are determined by a finite jet at a given boundary point and their centers 
 fill an open set. As a consequence, we obtain a finite jet determination result for this class of 
 smooth hypersurfaces.
\end{abstract} 

\title[Stationary discs and finite jet determination]{Stationary discs for smooth hypersurfaces of finite type and finite jet determination}

\author{Florian Bertrand and Giuseppe Della Sala}


\thanks{Research of the first author was supported by Austrian Science Fund FWF grant M1461-N25.}
\thanks{Research of the second 
author was supported by Austrian Science Fund FWF grant P24878 N25.}
\maketitle 

\section{Introduction}
In the celebrated paper \cite{le}, L. Lempert introduced stationary discs 
as holomorphic discs attached to a given hypersurface, admitting a meromorphic lift to the 
cotangent bundle with at most one pole of order one at $0$, and attached to the conormal bundle (see also 
\cite{pa, tu}). He proved that, on a strongly convex domain,  stationary discs  coincide with extremal discs for the 
Kobayashi metric, and constructed an analogue of the Riemann mapping in higher dimension 
(see also \cite{su-tu,co-ga-su}). Due to their geometric properties, stationary discs  are  
natural invariants to study pluripotential theory \cite{le}, 
mapping extension problems \cite{le2, tu}, and jet determination problems \cite{be-bl}. 
Their existence relies on the study of a nonlinear Riemann-Hilbert  boundary problem,
whose solvability is strongly related to nondegeneracy properties of the given hypersurface.  
In such a case, its associated conormal bundle is totally real \cite{we}, and one can construct small stationary discs nearby 
a given one by 
developing a perturbation theory by introducing some integers invariant under homotopy, namely the partial indices and 
the Maslov index (see \cite{gl1, gl2, ce, mcd-sa}).
However, under no nondegeneracy conditions, it is not clear whether or not one can ensure the existence of smooth stationary 
discs, and not much seems to be known in that direction.

Surprisingly, when one allows the unique pole of the meromorphic lift to be of greater order, there might exist 
a lot of such discs, which still form a biholomorphically invariant family. 
Based on this observation, we define, for a given hypersurface of finite type, the notion 
of $k_0$-stationary discs, as attached holomorphic discs admitting a meromorphic lift with at most one pole of 
order $k_0$ at $0$.
In order to construct $k_0$-stationary discs nearby a given one, the associated 
Riemann-Hilbert problem seems to be no longer relevant since the conormal bundle of the given hypersurface is not 
anymore totally real. Instead, we introduce a nonlinear operator, whose construction is essentially based on 
Toeplitz and Hankel operators, and whose linearization involve Fredholm operators. The properties of the associated Fredholm indices such as their invariance under homotopy,  ensure the existence of nearby small $k_0$-stationary discs attached to perturbed hypersurface, and the number of real variables  parametrizing  the perturbed discs is completely determined by those indices. 
 Our main theorem can be stated as follows:
\begin{theorem}\label{theo1}
Let $S=\{-\real w + P(z,\overline z)=0=0\}$ be a hypersurface, where 
$$P(z,\overline z)=\sum_{j=d-k_0}^{k_0} \alpha_{j} z^j \overline{z}^{d-j},  \ \ \alpha_{j} = \overline \alpha_{d-j}, \ \ \alpha_{k_0}\neq 0, \ \ \frac{d}{2}\leq k_0 \leq d-1$$
is a subharmonic homogeneous polynomial
satisfying the open condition $\{P_{z\overline z}=0\} = \{0\}$. For a class of smooth perturbations $M$ of $S$, there exists a finitely dimensional biholomorphically invariant Banach manifold of $k_0$-stationary discs attached to $M$.  
\end{theorem}
The precise perturbation of the model hypersurface $S$ that one can allow is explicitly described in Section 3.
We emphasize also that the dimension of the Banach manifold we construct 
depends only  on the type and the essential type of the model hypersurface $S$. 

\vspace{0.5cm}

Similarly to \cite{be-bl}, the discs constructed in Theorem \ref{theo1} are particularly adapted 
to study jet determination of biholomorphic mappings. As we know from results of 
\cite{eca, ch-mo, ta}, germs of biholomorphisms preserving a real-analytic Levi nondegenerate 
hypersurface $M$ in $\C^n$ are uniquely determined by their $2$-jet at $p \in M$. 
In $\C^2$, P. Ebenfelt, B. Lamel and D. Zaitsev  \cite{eb-la-za}  proved that 
finite jet determination  of germs at a point $p$ of local biholomorphisms preserving a real analytic hypersurface $M$ 
holds if and only if $M$ is not Levi flat at $p$, and obtained a $2$-jet determination for real-analytic 
hypersurfaces of finite type (see also \cite{ber5, ba-mi-ro, la-mi, ko-me}).  
In the present paper, we consider the situation of smooth  hypersurfaces of finite type in $\C^2$. 
We first point out that the finite jet determination of smooth CR mappings of  smooth CR 
submanifolds of finite type has been considered in \cite{eb, ki-za} (see also \cite{eb-la}). 
Their approach is based on the method of  
complete differential systems, introduced by C.K. Han \cite{ha1, ha2}, and the order of the jet determination  obtained in 
\cite{ki-za} is independent of the type. However, the assumption of finite nondegeneracy, 
as well as the ($\mathcal C^\infty$) smoothness of both the mapping and the 
submanifold, seem to be essential for their method to work. Recently, the first co-author and L. Blanc-
Centi \cite{be-bl} 
obtained a $2$-jet determination for merely $\mathcal C^4$ Levi nondegenerate hypersurfaces in 
$\C^n$ by considering invariance and geometric properties of stationary discs attached to such 
hypersurfaces; 
see also \cite{bu-kr, hu1} for a somehow similar approach in order to study boundary versions of 
H. Cartan's uniqueness theorem. Following this approach, we obtain:
\begin{theorem}\label{theo2}
Let $M \subset \C^2$ be a smooth pseudoconvex hypersurface of finite type
whose defining function is written as (\ref{eqdef3}) at $p=0 \in M$. 
Then there exists an integer $\ell$, depending only on  the type and the essential type of $M$, such that 
the germs at $p$ of biholomorphisms $H$ such that $H(M)=M'$ are uniquely determined by their $\ell$-jet at $p$.
\end{theorem}
The precise smoothness required for the hypersurface $M$ is formulated in Section $5$ and the precise order of jet is 
stated in Theorem \ref{theojet}. When the essential type of $M$ 
is exactly half of the type, namely when
$$M=\{- \real w + |z|^d+O\left(|z|^{d+1}+|\imag w||z|^{d-1}+|\imag w|^2\right)=0\},$$ 
we obtain the $2$-jet determination of germs of biholomorphisms at $0$. Finally, we stress that, due to the method we use, the result in Theorem \ref{theo2} holds for smooth CR-diffeomorphisms of $M$ as well (see Theorem \ref{theocr}), and we also 
obtain a boundary version of H. Cartan's uniqueness theorem (see Theorem \ref{theobd}). In principle, one might expect a higher jet determination for CR-diffeomorphisms than for biholomorphisms, although they coincide in the real-analytic case \cite{hu2, di-pi, di-pi2}. 
Thus, the question whether or not the order of the 
jet determination we obtain in Theorems \ref{theojet} is optimal, seems to be linked with the possible 
existence of one-sided biholomorphisms that do not extend across the hypersurface $M$. 
 
\vspace{0.5cm}

The paper is organized as follows. In Section 2, we discuss the properties of the Banach spaces of functions and operators that we
will need, and we define the notion of $k_0$-stationary discs. Section 3 is devoted to the construction of 
$k_0$-stationary discs
In Section 4, we study the geometric properties 
of $k_0$-stationary discs. Finally, we prove Theorem \ref{theo2} and other finite jet determination results in Section 5.

\section{Preliminaries}
We denote 
by $\Delta$ the  unit disc in $\C$.
We denote by $(z,w)$ the standard coordinates in $\C^2$. 
\subsection{Spaces of functions}

Let $k$ be an integer and let $0< \alpha<1$.
We denote by $\mathcal C^{k,\alpha}=\mathcal C^{k,\alpha}(b\Delta,\R)$ the space of real-valued functions  defined on $b\Delta$ of class 
$\mathcal{C}^{k,\alpha}$. The space  $\mathcal C^{k,\alpha}$ is endowed with  its usual norm
$$\|f\|_{\mathcal{C}^{k,\alpha}(\partial\Delta)}=\sum_{j=0}^{k}\|f^{(j)}\|_\infty+
\underset{\zeta\not=\eta\in b\Delta}{\mathrm{sup}}\frac{\|f^{(k)}(\zeta)-f^{(k)}(\eta)\|}{|\zeta-\eta|^\alpha},$$
where $\|f^{(j)}\|_\infty=\underset{\partial\Delta}{\mathrm{max}}\|f^{(j)}\|$.
We set $\mathcal C_\C^{k,\alpha} = \mathcal C^{k,\alpha} + i\mathcal C^{k,\alpha}$. Hence $f\in \mathcal C_\C^{k,\alpha}$ if and only if 
$\real f, \imag f \in \mathcal C^{k,\alpha}$. The space $\mathcal C_\C^{k,\alpha}$ is equipped with the norm
$$\|f\|_{\mathcal{C}_{\C}^{k,\alpha}(\partial\Delta)}=
\|\real f\|_{\mathcal{C}^{k,\alpha}(\partial\Delta)}+\|\imag f\|_{\mathcal{C}^{k,\alpha}(\partial\Delta)}$$ 
We denote by $\mathcal A^{k,\alpha}$ the subspace of $\mathcal C_{\C}^{k,\alpha}$ consisting of functions $f:\overline{\Delta}\rightarrow \C$, holomorphic on $\Delta$ with trace on 
$b\Delta$ belonging to $\mathcal C_\C^{k,\alpha}$. 
We define $(1-\zeta)\mathcal A^{k,\alpha}$
to be the subspace of 
$\mathcal C_{\C}^{k,\alpha}$  
of functions $f$ that can be written as $f=(1-\zeta)\tilde{f}$, with  $\tilde{f}\in \mathcal A^{k,\alpha}$. 
We equip  $(1-\zeta)\mathcal A^{k,\alpha}$ with the following norm 
$$\|(1-\zeta)\tilde{f}\|_{(1-\zeta)\mathcal A^{k,\alpha}}
=\|\tilde{f}\|_{\mathcal{C}_{\C}^{k,\alpha}(\partial\Delta)}.$$
Hence   $(1-\zeta)\mathcal A^{k,\alpha}$ is a Banach space. 
Notice that the inclusion of $(1-\zeta)\mathcal A^{k,\alpha}$ into 
$\mathcal A^{k,\alpha}$ is a bounded linear operator. We also point out that  
the linear operator  $L :  \mathcal A^{k,\alpha} \rightarrow (1-\zeta) \mathcal A^{k,\alpha}$, defined by 
$$L(\tilde{f})=(1-\zeta)\tilde{f}$$ 
is an isometry. 

Finally, we denote by $\mathcal C_0^{k,\alpha}$ the subspace of $\mathcal C^{k,\alpha}$ consisting of elements that can be written as $(1-\zeta)v$ with $v\in \mathcal C_\C^{k,\alpha}$. We equip $\mathcal C_0^{k,\alpha}$ with the norm
$$\|(1-\zeta)v\|_{\mathcal C_0^{k,\alpha}}=\|v\|_{\mathcal C_\C^{k,\alpha}}.$$
Notice that $\mathcal C_0^{k,\alpha}$ is a Banach space.

\subsection{Pseudoconvex hypersurfaces of finite type}

In this section, we recall some facts about pseudoconvex hypersurface of finite type in $\C^2$.
Let $M=\{r=0\}$ be a smooth pseudoconvex hypersurface defined in a neighborhood of the origin $C^2$. 
\begin{definition}\label{defco} 
Let $f~: \left(\Delta,0\right)\rightarrow \left(\C^2,0\right) $ be a holomorphic disc satisfying $f\left(0\right)=0$.
The order of contact $\delta_0\left(M,f\right)$ with $M$ at the origin is the degree of the first term 
in the Taylor expansion of $r \circ f$. We denote by $\delta\left(f\right)$ the multiplicity of $f$ at the origin. 
\end{definition}

We now define the D'Angelo type and the regular type of the real hypersurface $M$ at the origin.  

\begin{definition}\label{deftyp}\mbox{ }

 \begin{enumerate}[i)]
\item The D'Angelo type  of $M$ at the origin is defined by:
$$\Delta^1\left(M,0\right):=\sup \left\{\frac{\delta_0\left(M,f\right)}{\delta\left(f\right)}, \mbox{ } 
f:(\Delta,0)\rightarrow (\C^2,0), \mbox{ holomorphic} \right\}.$$
The point $0$ is a point of finite D'Angelo type $d$ if $\Delta^1\left(M,0\right)=d<+\infty$.
\item The   regular type of $M$ at origin is defined by:
\begin{eqnarray*}
\Delta^1_{{\rm reg}}\left(M,0\right)&:=&\sup \left\{\delta_0\left(M,f\right), \mbox{ } 
f:(\Delta,0)\rightarrow (\C^2,0),\mbox{ holomorphic}, d_0u\neq 0 \right\}.
\end{eqnarray*}
\end{enumerate}
\end{definition}
The  type condition  as defined in part $1$ of Definition \ref{deftyp} was introduced by 
J.-P. D'Angelo \cite{d'a} who proved that this coincides  with the regular type in complex manifolds
of dimension two.
Following \cite{ca} and \cite{fo-si}, a smooth pseudoconvex hypersurface $M$ of finite type $d$  
can be locally written as  $M=\{r=0\}$ with
\begin{equation}\label{eqdef}
r(z,w)=- \real w + P(z,\overline z)+O\left(|z|^{d+1}+|\imag w||z|^{\frac{d}{2}+1}+|\imag w|^2\right),
\end{equation}
where $P$ is a subharmonic homogeneous polynomial of degree $d$, not identically zero, with no harmonic terms. Notice then 
that $d$ is necessarily even.
In this paper, we will restrict our attention  smooth pseudoconvex hypersurfaces $M=\{r=0\}$ of finite type $d$, where
\begin{equation}\label{eqdef2}
r(z,w)=- \real w + P(z,\overline z)+O\left(|z|^{d+1}+|\imag w||z|^{d-1}+|\imag w|^2\right),
\end{equation}
such that  $\{P_{z\overline z}=0\} = \{0\}$. Notice that the set of points in $M$ of type greater than two is contained in $M\cap \{z=0\}$.
Given a polynomial $P$ satisfying the above conditions, the following hypersurface  will be called  a {\it model hypersurface},  $S = \{\rho = 0\}$ with  
$$\rho(z,w)=- \real w + P(z,\overline z).$$
To fix some notations, we write 
$$P(z,\overline z)=\sum_{j=d-k_0}^{k_0} \alpha_{j} z^j \overline{z}^{d-j},  \ \ \alpha_{j} = \overline \alpha_{d-j}, \ \ \alpha_{k_0}\neq 0.$$
where $d/2\leq k_0 \leq d-1$. The integer $d-k_0$ is {\it the essential type} of the model hypersurface. 
\vspace{0,7cm}

\subsection{$k_0$-stationary discs}

Let $M=\{r=0\}$ be a smooth pseudoconvex hypersurface defined in a neighborhood of the origin in $\C^2$. 
Let $k$ be an integer and let $0<\alpha<1$.     
A holomorphic disc $f \in (\mathcal A^{k,\alpha})^2$ is  {\it attached} to $M$  if $f(b\Delta)\subset M$.

\begin{definition}\label{defstat}
A holomorphic disc $f \in (\mathcal A^{k,\alpha})^2$ attached to  $M=\{r=0\}$ is $k_0$-stationary if there exists a  continuous function 
$c:b\Delta\to\R^*$ such that the map $\zeta^{k_0} c(\zeta)\partial r(f(\zeta))$, defined on $b \Delta$, 
extends as a map in $(\mathcal A^{k,\alpha})^2$. The map $(c,f)$ is called a lift of $f$.  
\end{definition}

We first prove that such holomorphic discs are invariant under biholomorphisms.
\begin{prop}\label{statinvar}
Let $M \subset \C^2$ be a smooth pseudoconvex real hypersurface  of finite type containing $0$. 
Let $H$ be a local biholomorphism in $\C^2$ sending $M$ to a real hypersurface $M'$. 
If  $f: \Delta \rightarrow \C^2$ is a $k_0$-stationary disc
attached to $M$ , then the disc $H\circ f$ is a  $k_0$-stationary disc attached to $M'$.   More precisely, if $(c,f)$ is a lift of $f$ attached to  $M$
then $(c,H\circ f)$ is a lift of $H\circ f$ attached to  $M'$. 
\end{prop}

\begin{proof}
Let $f:\Delta \rightarrow \C^2$ be a $k_0$-stationary disc attached to $M=\{r=0\}$.  
and let $c:b\Delta\to\R^*$ be such that $\zeta^{k_0} c(\zeta)\partial r(f(\zeta))$ extends as a map in 
$(\mathcal A^{k,\alpha})^2$. Since $H$ sends $M$ to $M'$, the function $r \circ H^{-1}$ is a (local) defining function for $M'$.
It follows that $\zeta^{k_0} c(\zeta)\partial (r\circ H^{-1})(H\circ f(\zeta))=\zeta^{k_0} c(\zeta)\partial r (f(\zeta))(\partial H(f(\zeta))^{-1}$ extends as well as a $(\mathcal A^{k,\alpha})^2$ map.  
\end{proof}

We denote by $\mathcal{S}^{k_0,r}$ the set of lifts of $k_0$-stationary discs attached to 
$M=\{r=0\}$. Moreover, assume that the 
local definition function of $M$ is written in the form (\ref{eqdef}). We denote by $\mathcal{S}_{0}^{k_0,r}$, or simply $\mathcal{S}_{0}^r$, the elements $(c,f) \in \mathcal{S}^{k_0,r}$ satisfying $f(1)=0$. In such a case, we say that the lift $(c,f)$ is {\it 
tied to the origin}.

\begin{example}\label{exmodel}
Consider a model hypersurface $S = \{\rho = 0\}$ with  
$$\rho(z,w)=- \real w + P(z,\overline z)=- \real w + \sum_{j=d-k_0}^{k_0} \alpha_{j} z^j \overline{z}^{d-j}.$$
We have 
$$
\left\{
\begin{array}{lll}
\frac{\partial \rho}{\partial w}(z,w) = -\frac{1}{2}\\
\\
\frac{\partial\rho}{\partial z}(z,w) = P_z(z,\overline z) = \sum_{j=d-k_0}^{k_0} j\alpha_{j} z^{j-1} \overline{z}^{d-j}.\\
\end{array}
\right.
$$

We restrict to coefficient functions of the form $c(\zeta) = \left(\overline{b\zeta} +1+b\zeta\right)^{k_0}$ where $b \in \C$ is such that 
$|b|<1/2$. Set  $c'(\zeta) = \overline{b\zeta} +1+b\zeta$. It follows that $\zeta c'$, and thus 
$-\frac{1}{2} \zeta^{k_0} c,
$ extend holomorphically to the unit disc.
Following the arguments in \cite{be-bl}, we set 
$$h(\zeta)=\frac{1-\zeta}{1-a(b)\zeta}v,$$ 
with $a= \frac{-1+\sqrt{1-4|b|^2}}{2b}$ and  $v \in \C^*$, which implies that $\zeta c' \overline h\in \mathcal A^{k,\alpha}$.  Then 
$$\zeta^{k_0} c(\zeta) P_z (h(\zeta),\overline{h(\zeta)}) = \sum_{j=d-k_0}^{k_0} j\alpha_{j} (\zeta c'(\zeta))^{k_0 - d + j}  h(\zeta)^{j-1} (\zeta c'(\zeta)\overline{h(\zeta)})^{d-j}.$$ 
Since $k_0 - d + j \geq 0$ and  $d - j\geq 0$ for $d-k_0\leq j \leq k_0$, every term in the sum on the right hand side belongs to 
$\mathcal A^{k,\alpha}$. Hence $\zeta^{k_0}cP_z(h,\overline h) \in \mathcal A^{k,\alpha}$. 
Imposing the further condition $g(1)=0$, one can find, by standard results about the 
Hilbert transform, a map $g$ such that 
$f=(h,g)$ is a $k_0$-stationary disc attached to $S$ and tied to the origin.

In particular, if $b=0$ and $v=1$, it follows that $c_0$ is identically equal to $1$ and that 
$h_0(\zeta)=1-\zeta$. Denote by $f_0=(h_0,g_0)$ the corresponding 
stationary disc.    
\end{example}
The disc $f_0$ we have given in Example \ref{exmodel} is essential in our approach in order  to obtain a family of $k_0$-stationary discs by deformation of both $f_0$ and the model hypersurface.

\subsection{Regularity and Fredholmness of integral operators on $b\Delta$}

We denote by $H^2(b\Delta)$ the classical Hardy space on the unit circle and by $\mathcal P':L^2(b\Delta)\to H^2(b\Delta)$ the Szeg\"o projection. The Szeg\"o projection defines a linear operator $\mathcal P': \mathcal C_\C^{k,\alpha}\to \mathcal A^{k,\alpha}$. 
Denoting by $\mathcal T$ the Hilbert transform, one  can deduce by \cite{ber} (6.1.37) the following relation 
\begin{equation}\label{eqsh}
-2i\mathcal P' = \mathcal T - i{\rm Id}-iC_0, 
\end{equation} 
where $C_0u=\int_0^{2\pi}u(e^{i\theta})d\theta$.
Due to  Privalov's Theorem,  the Hilbert transform  is a well-behaved operator (see \cite{ber}, Corollary $6.1.31$): 
\begin{theorem}
Let $k\geq 0$ be an integer and let $0<\alpha<1$ be a real number. The Hilbert transform 
$\mathcal T: \mathcal{C}^{k,\alpha} \rightarrow \mathcal{C}^{k,\alpha}$ is a bounded 
linear operator. 
\end{theorem}
Hence, it follows from (\ref{eqsh}) that the Szeg\"o projection $\mathcal P': \mathcal C_\C^{k,\alpha}\to \mathcal A^{k,\alpha}$ is a 
bounded linear operator. Out of convenience, however, we will mainly consider the analogous bounded projection $\mathcal P:\mathcal C_\C^{k,\alpha}\to 
\overline{\zeta\mathcal A}^{k,\alpha}$ defined by 
\begin{equation}
\mathcal P(u) = \overline {\zeta\mathcal P'\left(\overline {\zeta u}\right)}=u-\mathcal P'(u).
\end{equation}
The space $\overline{\zeta\mathcal A}^{k,\alpha}$ is endowed with the induced norm of  $\mathcal C_\C^{k,\alpha}$.
Notice that a function $u$ extends as a function  in $\mathcal A^{k,\alpha}$ if and only if $\mathcal P(u) =0$. 
More precisely, if $u = u' + u''$ with $u' = \sum_{n\geq 0} u_n\zeta^n$ and 
 $u'' = \mathcal P(u) = \sum_{n< 0} u_n \zeta^n$ then $\mathcal P(u)=u''$.

Recall that when $\varphi$ is a  complex valued continuous function defined on $b\Delta$ that   
does not vanish,  then the Toeplitz operator with symbol $\varphi$ defined by 
$T_{\varphi}=\mathcal P'(\varphi .): H^2(b\Delta) \rightarrow H^2(b\Delta)$ is a Fredholm operator, that is with finite dimensional kernel and cokernel. 
Moreover, it index, namely $dim_\C \ker T_{\varphi}-dim_\C \  {\rm coker} T_{\varphi}$ ,  is the opposite of the  winding number 
$-\frac{1}{2\pi i }\int_{b\Delta} \frac{d\varphi}{\varphi}$ of $\varphi$. It follows then that for such a $\varphi$, the operator  $\mathcal P(\varphi .): 
\overline{H}^2(b\Delta) \rightarrow \overline{\zeta H}^2(b\Delta)$ is also Fredholm of index $-\frac{1}{2\pi i }\int_{b\Delta} 
\frac{d\overline{\zeta\varphi}}{\overline{\zeta\varphi}}$. Note that, although the operator $\mathcal P(\varphi .)$ is defined in a similar way than the Hankel 
operator $\mathcal P(\varphi .): H^2(b\Delta) \rightarrow \overline{\zeta H}^2(b\Delta)$, they do not agree since their source spaces differ. In this paper, we need to consider the operator $\mathcal P(\varphi .)$ restricted to 
$\overline{\mathcal A}^{k,\alpha} \subset \overline{H}^2(b\Delta)$. Although it is not clear whether or not
 $\mathcal P(\varphi .)$ is Fredholm in that case, for a very special class of functions 
$\varphi$, we have the following:
\begin{lemma}\label{lemfred}
Let $\varphi \in \mathcal C^{k,\alpha}$ be a holomorphic polynomial whose zeros are contained in the unit disc $\Delta$. 
Then the operator $\mathcal P(\varphi .): \overline{\mathcal A}^{k,\alpha} \rightarrow \overline{\zeta\mathcal A}^{k,\alpha}$ 
is Fredholm of index the opposite of the  winding number of $\overline{\zeta \varphi}$. 
More precisely,  $\mathcal P(\varphi .)$ is surjective and its kernel has complex dimension $-\frac{1}{2\pi i }\int_{b\Delta} \frac{d\overline{\zeta\varphi}}{\overline{\zeta\varphi}}$. 
\end{lemma}
\begin{proof}
First note that since  $\overline{\mathcal A}^{k,\alpha}\subset \overline{H}^2(b\Delta)$, the kernel of  $\mathcal P(\varphi .): \overline{\mathcal A}^{k,\alpha} \rightarrow \overline{\zeta\mathcal A}^{k,\alpha}$ is included in the one of $\mathcal P(\varphi .): 
\overline{H}^2(b\Delta) \rightarrow \overline{\zeta H}^2(b\Delta)$  and hence is finitely dimensional. 
Write 
$$\varphi(\zeta) = C\prod_{1\leq j \leq \ell} (q_j - \zeta)=\frac{C}{\overline{\zeta}^\ell}\varphi_2(\zeta),$$ 
where 
$\varphi_2(\zeta)=\prod_{1\leq j \leq \ell} (\overline{\zeta}q_j - 1)$.   
Then $\varphi_2$ extends antiholomorphically to the unit disc and the extension is nowhere vanishing on $\overline \Delta$. It follows that both $\varphi_2$ and its inverse belong to $\overline{\mathcal A}^{k,\alpha}$. Now, for any $v\in \overline{\zeta\mathcal A}^{k,\alpha}$ define 
$$u(\zeta) = \frac{\overline \zeta^{\ell}}{C} \frac{v(\zeta)}{\varphi_2(\zeta)}\in \overline{\mathcal A^{k,\alpha}}.$$ 
Since $v\in \overline{\zeta\mathcal A}^{k,\alpha}$, we have $\mathcal P(\varphi u)=v$, which proves the surjectiviy of $\mathcal P(\varphi .)$. Thus $\mathcal P(\varphi .)$ is a Fredholm operator and 
its index is equal to $dim_\C \ker \mathcal P(\varphi.)$.

In order to compute its Fredholm index, we need the following observation. The boundedness of 
the operator $\mathcal P$ implies that the map $\varphi \mapsto \mathcal P(\varphi .)$ is continuous. 
Due to the local constancy of the index of a Fredholm operator (see Theorem 5.2 p. 42 \cite{bo-bl}), it follows 
that if  $\varphi$ and $\varphi'$ are two homotopically equivalent holomorphic polynomials whose zeros are contained in the unit disc $\Delta$, then  $\mathcal P(\varphi .)$ and  
$\mathcal P(\varphi .)$ have the same index. And since any such function $\varphi: b\Delta \to \C^*$ is homotopically equivalent to some $\zeta^m$, for some integer $m\geq 0$, it is enough to compute the 
index of the Fredholm operator   $\mathcal P(\zeta^m .)$. 
We  write the Fourier expansion of $u \in \overline{\mathcal{A}}^{k,\alpha}$ as 
$u(\zeta) =  \sum_{n < 0 } u_n \zeta^n$. Thus 
$$\mathcal P(\zeta^m u)= \sum_{n<-m} u_n \zeta^{m+n}=0$$
if and only if $u_n=0$  whenever $n< -m$. This implies that  
$dim_\C \ker\mathcal P(\zeta^m .)=m+1$. 
\end{proof}

Finally, we will need the following two lemmas.

\begin{lemma}\label{lemc}
Let $k_0 \geq 1$ be an integer and let $c:b\Delta\to\R^*$ be a continuous function. Then 
$\mathcal P(\zeta^{k_0}c)=0$ if and only if 
\begin{equation}\label{eqc}
c(\zeta) = \overline c_{k_0} \zeta^{-k_0} + \overline c_{k_0 - 1} \zeta^{-(k_0-1)} + \ldots + c_0+ \ldots + c_{k_0-1} \zeta^{k_0-1} + 
c_{k_0} \zeta^{k_0},
\end{equation}
 where $c_0 \in \R$ and $c_1,\ldots,c_{k_0} \in \C$.   
\end{lemma}

\begin{proof}
We  write the Fourier expansion of $c:b\Delta\to\R^*$  as $c(\zeta) =  \sum_{n\in \mathbb Z} c_n \zeta^n$, where $c_n\in \mathbb C$ and 
$c_n = \overline c_{-n}$ for all $n\in \mathbb Z$. Since the Fourier expansion of $\zeta^{k_0} c$ is given by $\sum_{n\in \mathbb Z} c_n \zeta^{n+k_0}$, in order for it to belong to $\mathcal A^{k,\alpha}$ we must have $c_n = 0$ for all $n< -k_0$, from which follows (\ref{eqc}). Conversely, the fact that such a $c$ satisfies $P(\zeta^{k_0}c)=0$ 
is immediate. 
\end{proof}

\begin{lemma}\label{Pu}
Let $u\in \mathcal C_\C^{k,\alpha}$ and let  $\ell\in \mathbb Z$. Then $\mathcal P\left((1-\zeta)^\ell u\right) = 0$ if and only if $\mathcal P u = 0$.
\end{lemma}
\begin{proof}
 By induction it is sufficient to show that $\mathcal P((1-\zeta) u) = 0$ if and only if $\mathcal P u = 0$.
Let $u(\zeta) =  \sum_{n\in \mathbb Z} c_n \zeta^n$ be the Fourier expansion of $u$, and write $u = u' + u''$ with $u' = \sum_{n\geq0} u_n\zeta^n$ and 
 $u'' = \mathcal P(u) = \sum_{n<0} u_n \zeta^n$. 
 It is clear that $(1-\zeta)u'$ always extends holomorphically, hence  we must show that $\mathcal P((1 - \zeta)u'') = 0$ if and only if $u'' = 0$. Now, we have $(1-\zeta)u'' = \sum_{n\leq 0} v_n \zeta^n$, where
$$
\left\{
\begin{array}{lll}
v_0=u_{-1}\\
\\
v_n = u_{n}-u_{n-1}\mbox{,  } \mbox{ } n\leq -1.
\end{array}
\right.
$$
If $\mathcal P((1-\zeta) u'')=0$, we must have $v_n=0$ for all $n\leq -1$, which implies that $u_n=u_{n-1}$ 
for all $n \leq -1$. Since $u''\in \mathcal C_\C^{k,\alpha}$, this is only possible if $u''=0$.
\end{proof}

\section{Construction of $k_0$-stationary discs}

Consider a model hypersurface $S= \{\rho = 0\}$  of type $d$ 
$$\rho(z,w)=- \real w + P(z,\overline z)=- \real w + \sum_{j=d-k_0}^{k_0} \alpha_{j} z^j \overline{z}^{d-j}.$$
Recall that in view of Example \ref{exmodel}, we consider the $k_0$-stationary disc attached to $S$ given by 
$f_0=(h_0,g_0)=(1-\zeta,g_0)$, with the coefficient function $c_0\equiv 1$. We fix $\delta>0$ such 
that $f_0(\Delta) \subset (2\delta\Delta)^2$. The aim of this section is to construct a finite dimensional manifold of lifts of $k_0$-stationary 
discs attached to small perturbations of the given model hypersurface. 
 
Let $k>0$ be an integer and let $0<\alpha<1$.  
Denote by $X$ the (affine) Banach space parametrizing the deformations that one consider to the model $S$. We define $X$ as the set of functions $r\in \mathcal  C^{k+3}\left(\overline {\delta\Delta}^2\right)$  which can be written as
\begin{equation}\label{eqdef3}
r(z,w) = \rho(z,w) +\theta(z,\imag w ),
\end{equation}
with 
$$\theta(z,\imag w )= \sum_{i+j=d+1} (z^i \overline z^j) \cdot r_{ij0}(z)+
\sum_{l=1}^{d-1}\sum_{i+j=d-l} z^i \overline z^j (\imag w)^l \cdot r_{ijl}(z,\imag w)+\theta_1(\imag w)$$ 
where $r_{ij0} \in  \mathcal C^{k+3}_\C\left(\overline{\delta\Delta}\right), 
r_{ijl}\in \mathcal  C^{k+3}_\C\left(\overline{\delta\Delta}\times[-\delta,\delta]\right)$, 
  $\theta_1 \in \mathcal C^{k+3}\left([-\delta,\delta]\right)$,
and where  $\theta_1(\imag w)=O(|\imag w|^2)$. Furthermore, we will consider the norm 
$$\|r\|_X = \sup \|r_{ijl}\|_{\mathcal C^{k+3}}+\|\theta_1\|_{\mathcal C^{k+3}},$$
 so that $X$ is isomorphic to a (real) closed subspace of 
$\mathcal  C^{k+3}_\C\left(\overline{\delta\Delta}\times [-\delta,\delta]\right)$, hence it is a Banach space. The inclusion of $X$ into 
$\mathcal  C^{k+3}\left(\overline {\delta\Delta}\times[-\delta,\delta]\right)$ is an (affine) linear bounded map, and 
in particular it is of class $\mathcal C^1$.
We define the set  
$$Y = \mathcal C^{k,\alpha} \times (1-\zeta)\mathcal A^{k,\alpha} \times 
(1-\zeta)\mathcal A^{k,\alpha},$$
and we equip it with the following norm 
$$\|(c,h,g)\|_Y=\|c\|_{\mathcal{C}^{k,\alpha}}+\|h\|_{(1-\zeta)\mathcal A^{k,\alpha}}+
\|g\|_{(1-\zeta)\mathcal A^{k,\alpha}}.$$ 

Theorem \ref{theo1} is a consequence of  the following more precise result: 
\begin{theorem}\label{theodiscs}
Let $S=\{\rho=0\}$ be a model hypersurface of finite type $d$. Then for any integer $k\geq 0$ and 
$0<\alpha<1$, there exist some open 
neighborhoods $V$ of $\rho$ in $X$  and $U$ of $0$  in $\R^{4k_0-d+3}$, $\eta>0$, and a map
$\mathcal{F}:V \times U \to  Y$
 of class $\mathcal{C}^1$, such that:
\begin{enumerate}[i)]
\item $\mathcal{F}(\rho,0)=(1,f_0)$,
\item for all $r\in V$, the map $\mathcal{F}(r,\cdot):U\to \{(c,f) \in Y
\ | \ (c,f) \in \mathcal{S}_{0}^r,\ \|(c,f)-(1,f_0)\|_Y<\eta\}$ 
is one-to-one and onto.
\end{enumerate}
\end{theorem}
In order to prove Theorem \ref{theodiscs}, we need to consider the zero set of the map $T=(T_1,T_2, T_3)$
\[T: X\times Y \to (\overline{\zeta\mathcal A}^{k,\alpha})^2 \times \mathcal C_0^{k,\alpha},\]
defined by
$$
\left\{
\begin{array}{lll}
T_1(r,c,h,g) =\mathcal P \left(\zeta^{k_0} c \frac{\partial r}{\partial z}(h,g) \right)\\
\\
T_2(r,c,h,g) =\mathcal P \left(\zeta^{k_0} c \frac{\partial r}{\partial w}(h,g) \right)\\
\\
T_3(r,c,h,g) = r(h,g)
\end{array}
\right.
$$
in a neighborhood of $(\rho, c_0, h_0, g_0)$. For any fixed $r\in X$, the zero set of $T(r,\cdot)$ is the set  $\mathcal{S}_{0}^r$ of $k_0$-stationary discs attached to $\{r=0\}$ and tied to the origin.

Although the map $T$ is of class $\mathcal C^1$, its derivative with respect to $Y$ is, in general, not
surjective. Our aim is then to replace it with another 
map $T'$ having the same zero set as 
$T$ and whose derivative with respect to $Y$ is surjective, hence allowing the application of the implicit function theorem.

We first  define a polynomial $Q$ in such a way that 
\begin{equation}\label{eqQ}
\zeta^{k_0}P_{z \overline z}(1-\zeta, 1-\overline \zeta) = (\zeta - 1)^{d-2}Q(\zeta)
\end{equation} 
for all $\zeta\in b\Delta.$ The polynomial $Q$ can be explicitly computed: 
$$
\begin{array}{lll}
\zeta^{k_0}P_{z\overline z}(1-\zeta,1-\overline\zeta) &= & \sum_{j=(d-k_0)}^{k_0} \gamma_j(1- \zeta)^{j-1} \zeta^{k_0}(1- \overline\zeta)^{d-j-1}\\
\\
&= &  ( \zeta - 1)^{d-2}\sum_{j=d-k_0}^{k_0}(-1)^{j-1}\gamma_j  \zeta^{k_0 + j + 1 - d}\\
\\
& = &( \zeta - 1)^{d-2}Q(\zeta)
\end{array}
$$
where $\gamma_j = j(d-j)\alpha_{j}$. The assumption made on $P$ implies that $Q$ does not have any roots on the unit circle 
$b\Delta$; also note that $Q$ has always one root of multiplicity one at the origin. 
Denote by $q_j$, $1\leq j \leq  2k_0 + 1 - d$, the roots of $Q$, 
and suppose that $q_1, \ldots, q_{i_0}$ are the ones lying outside of $\Delta$ and $q_{2k_0 + 1 - d} = 0$. We set $s(\zeta)= \prod_{1\leq j \leq i_0} (q_j - \zeta)$, $t(\zeta) = \prod_{i_0 + 1\leq j \leq 2k_0 - d} (q_j - \zeta)$, and write
 \begin{equation}\label{eqQ2}
Q(\zeta) = C\zeta s(\zeta)t(\zeta)
\end{equation} 
for some constant $C\in \C$. Put $\ell_0 = 2k_0 - i_0 - d$, so that $Q(\zeta)$ admits exactly $\ell_0 + 1$ roots inside the unit disc and $i_0$ roots outside.

\begin{lemma} We have $\ell_0 = i_0= k_0 - d/2$.
\end{lemma}
\begin{proof}
For the given $k_0\in \mathbb N$, consider the set $P_{k_0,d}$ of the real polynomials $P'(z,\overline z)$ such that 
\[P'(z,\overline z)=\sum_{j=d-k_0}^{k_0} \alpha'_{j} z^j \overline{z}^{d-j},  \ \ \alpha'_{j} = \overline \alpha'_{d-j}\]
and such that $P'_{z\overline z}$ vanishes exactly at $0$ and is positive elsewhere. For each $P'\in P_{k_0,d}$, denote by $Q_{P'}$ the polynomial defined above. The computation performed before the lemma shows that the coefficients of $Q_{P'}$ depend continuously on those of $P'$, and that $Q_{P'}$ never vanishes on the unit circle. Thus, the argument principle implies that the number of zeroes of $Q_{P'}$ lying inside the unit disc is constant on any connected component of $P_{k_0,d}$. On the other hand, it is easy to check that $P_{k_0,d}$ is a convex cone in the space of the real polynomials of degree $d$, hence it is connected.

Choosing now $P'(z,\overline z) = |z|^d$, we obtain $Q_{P'}(\zeta) = (-1)^{\frac{d}{2} - 1} \frac{d^2}{4} \zeta^{k_0 + 1 - \frac{d}{2}}$, which vanishes at $z=0$ with multiplicity $k_0 + 1 - d/2$. It follows that the polynomial $P$ also admits exactly $k_0 + 1 - d/2$ zeroes inside the unit disc, hence $\ell_0 = k_0 - d/2$ as claimed.
\end{proof}

We now define the map $T' = (T_1',T_2', T_3')$ by putting $T_2' = T_2$, $T_3'= T_3$ and
\[T_1'(r,c,h,g) =\mathcal P \left(\frac{\zeta^{k_0}}{(1-\zeta)^{d-1} s(\zeta)}\cdot c \frac{\partial r}{\partial z}(h,g) \right).\]
The map $T_1'$  is well-defined because the choice of the space $X$ implies 
$\frac{\partial r}{\partial z}(h,g)\in (1-\zeta)^{d-1}\mathcal C_\C^{k,\alpha}$ for all $r\in X$ and $h,g\in 
(1-\zeta)\mathcal A^{k,\alpha}$. 
The fact that the zero set of $T_1'$ is the same as the one of $T_1$ follows from Lemma \ref{Pu} and 
from the fact that both $s(\zeta)= \prod_{1\leq j \leq i_0} (q_j - \zeta)$ and its inverse extend 
holomorphically to the unit disc, which implies that $s \cdot u \in \mathcal A^{k,\alpha}$ if and only if $ 
u \in \mathcal A^{k,\alpha}$.

\begin{lemma}
 The map $T'$ is of class $\mathcal C^1$.
\end{lemma}
\begin{proof}
Since $\mathcal P:\mathcal C_\C^{k,\alpha}\to \overline{\zeta\mathcal A}^{k,\alpha}$ is a bounded linear operator, it is enough to 
show that the following map 
\[\tilde{T'}=(\tilde{T_1'},\tilde{T_2'}, T_3'): X\times Y \to (\mathcal{C}_\C^{k,\alpha})^2 \times \mathcal C_0^{k,\alpha},\]
defined by
$$
\left\{
\begin{array}{lll}
\tilde{T_1'}(r,c,h,g) = \frac{\zeta^{k_0}}{(1-\zeta)^{d-1} s(\zeta)}\cdot c \frac{\partial r}{\partial z}(h,g)\\
\\
\tilde{T_2'}(r,c,h,g) =\zeta^{k_0} c \frac{\partial r}{\partial w}(h,g)\\
\\
T_3'(r,c,h,g) = r(h,g)
\end{array}
\right.
$$
is of class $\mathcal{C}^1$.

According to Lemma 5.1 in \cite{hi-ta}, the map  $\tilde{T_2'}$ is of class $\mathcal{C}^1$ when considered as a map from 
 $X\times \mathcal C^{k,\alpha} \times \mathcal A^{k,\alpha} \times \mathcal A^{k,\alpha}$. 
 Then it follows from the boundedness of the inclusion of $(1-\zeta)\mathcal A^{k,\alpha}$ into 
$\mathcal A^{k,\alpha}$ noticed    
 in Subsection $2.1$ that  
  $\tilde{T_2'}: X\times Y \to \mathcal{C}_\C^{k,\alpha}$  
 is of class $\mathcal{C}^1$. 
 
To treat the smoothness of $\tilde T_1'$, we define, for every $i,j$ such that $i+j=d+1$, maps $S_{ij0}, U_{ij0}: 
X \to \mathcal C_{\C}^{k+2}(\overline{\delta\Delta}^2)$ of class $\mathcal C^1$ in the following way:
\[(S_{ij0}(r))(z,z') = iz^{i-1} \overline z^j \cdot r_{ij0}(z'),\] 
\[(U_{ij0}(r))(z,z') = z^i \overline z^j \cdot \frac{\partial r_{ij0}}{\partial z}(z').\]
We also define, for every $i,j$ and $1\leq l\leq d-1$  satisfying $i+j+l=d$, maps $S_{ijl}, U_{ijl}: 
X \to \mathcal C_{\C}^{k+2}(\overline{\delta\Delta}^3\times[-\delta,\delta])$ of class $\mathcal C^1$ by:
\[(S_{ijl}(r))(z,w, z', \imag w') = iz^{i-1} \overline z^j\left(w+\overline{\zeta w}\right)^l \cdot r_{ijl}(z',\imag w'),\] 
\[(V_{ijl}(r))(z,w,z', \imag w') = z^i \overline z^j \left(w+\overline{\zeta w}\right)^l  \cdot \frac{\partial r_{ijl}}{\partial z}(z',\imag w').\]
Posing $h = (1-\zeta)\tilde{h}$, $g = (1-\zeta)\tilde{g}$, a straightforward computation gives
$$
\begin{array}{lll}
\tilde T_1'(r,c,h,g) & = &  \frac{\zeta^{k_0}}{s(\zeta)}cP_z\left(\tilde{h},-\overline{\zeta\tilde{h}}\right) + 
\frac{\zeta^{k_0}(1-\zeta)}{s(\zeta)}c\sum_{i+j=d+1} \frac{1}{(-\zeta)^{j}}(S_{ij0}(r))(\tilde{h},h)\\
\\
& &+ \frac{\zeta^{k_0}(1-\zeta)^2}{s(\zeta)}c\sum_{i+j=d+1}  \frac{1}{(-\zeta)^{j}}(U_{ij0}(r))(\tilde{h},h)\\
\\
&& +\frac{\zeta^{k_0}}{s(\zeta)}c\sum_{l=1}^{d-1}\sum_{i+j=d-l} \frac{1}{(-\zeta)^{j}}(S_{ijl}(r))(\tilde{h},\tilde{g},h,\imag g)\\
\\
&& + \frac{\zeta^{k_0}(1-\zeta)}{s(\zeta)}c\sum_{l=1}^{d-1}\sum_{i+j=d-l}  \frac{1}{(-\zeta)^{j}}(U_{ijl}(r))(\tilde{h},\tilde{g},h, \imag g)\\
\end{array}
$$
Since, again by Lemma 5.1 in \cite{hi-ta} and the discussions made in Subsection $2.1$, the maps 
$X\times (1-\zeta)\mathcal A^{k,\alpha}\times (1-\zeta)\mathcal A^{k,\alpha}\to \mathcal C^{k,\alpha}_\C$ defined as 
$(r,h,g)\mapsto (S_{ij0}(r))(\tilde{h},h)$, $(r,h,g)\mapsto (U_{ij0}(r))(\tilde{h},h)$, 
$(r,h,g)\mapsto (S_{ijl}(r))(\tilde{h},\tilde{g},h,\imag g)$, $(r,h,g)\mapsto (U_{ijl}(r))(\tilde{h},\tilde{g},h,\imag g)$ 
are of class $\mathcal C^1$, it follows that $T_1'$ is in turn of class $\mathcal C^1$.

The proof that $T_3': X\times Y \rightarrow \mathcal C_0^{k,\alpha}$ is of class $\mathcal C^1$ is 
 analogous to (but simpler than) the proof for $\tilde T_1'$, hence we shall omit it.
 \end{proof}

We then show the following:
\begin{lemma}\label{surj}
The Banach space derivative $T'_Y$ is surjective at the point $p_0=(\rho, c_0, h_0, g_0)$.
\end{lemma}
\begin{proof}
Choosing any $(c',h',g')\in Y$, for any $p=(r,c,h,g)$ we can write 
$$T'_Y(p)[c',h',g'] = \left((T'_1)_Y(p)[c',h',g'], (T'_2)_Y(p)[c',h',g'], (T'_3)_Y(p)[c',h',g']\right),$$ where

$$
\left\{
\begin{array}{lll}
(T'_1)_Y(p)[c',h',g'] &= &\mathcal P \left(\frac{\zeta^{k_0}}{s(\zeta)(1-\zeta)^{d-1}}\cdot \left (c' \frac{\partial r}{\partial z}(h,g) + c h' \frac{\partial^2 r}{\partial z^2}(h,g) + c\overline h' \frac{\partial^2 r}{\partial z \partial \overline z}(h,g) \right) \right)\\
\\
&& + \mathcal P \left(\frac{\zeta^{k_0}}{s(\zeta)(1-\zeta)^{d-1}}\cdot \left (c g' \frac{\partial^2 r}{\partial z \partial w}(h,g) + c\overline g' \frac{\partial^2 r}{\partial z \partial \overline w}(h,g) \right) \right),\\
\\
(T'_2)_Y(p)[c',h',g'] &= & \mathcal P \left(\zeta^{k_0}\cdot \left (c' \frac{\partial r}{\partial w}(h,g) + c h' \frac{\partial^2 r}{\partial w\partial z}(h,g) + c\overline h' \frac{\partial^2 r}{\partial w \partial \overline z}(h,g) \right) \right)\\
\\
&& + \mathcal P \left(\zeta^{k_0}\cdot \left (c g' \frac{\partial^2 r}{\partial w^2}(h,g) + c\overline g' \frac{\partial^2 r}{\partial w \partial \overline w}(h,g) \right) \right),\\
\\
(T'_3)_Y(p)[c',h',g'] &= & g' \frac{\partial r}{\partial w}(h,g) + \overline g'\frac{\partial r}{\partial \overline w}(h,g) + h'\frac{\partial r}{\partial z}(h,g) + \overline h'\frac{\partial r}{\partial \overline z}(h,g).
\end{array}
\right.
$$
At $p_0= (\rho, c_0, h_0, g_0)$ this becomes
$$
\left\{
\begin{array}{lll}
(T'_1)_Y(p_0)[c',h',g'] &=& \mathcal P \left(\frac{\zeta^{k_0}}{s(\zeta)(1-\zeta)^{d-1}} \cdot \left ( P_{z^2}(1-\zeta, 1-\overline\zeta)h' + P_{z\overline z}(1-\zeta, 1-\overline\zeta)\overline h' \right ) \right )\\
\\
&&+  \mathcal P \left(\frac{\zeta^{k_0}}{s(\zeta)(1-\zeta)^{d-1}} P_z(1-\zeta, 1-\overline\zeta) c'  \right ),\\
\\
(T'_2)_Y(p_0)[c',h',g']  &= & \mathcal P \left(-\frac{\zeta^{k_0}}{2}c' \right ),\\
\\
(T'_3)_Y(p_0)[c',h',g'] &= &-\frac{g'}{2} - \frac{\overline g'}{2} + P_z(1-\zeta, 1-\overline\zeta) h' +  P_{\overline z}(1-\zeta, 1-\overline\zeta) \overline h'.
\end{array}
\right.
$$
Because of the triangular form of the previous expressions we see that it is sufficient to show that the derivative of $T'_2$ in the direction of $c'$ is surjective (onto $\overline{\zeta\mathcal A}^{k,\alpha}$) and that the same is true for the derivatives of $T'_1$ in the $h'$ direction and of $T'_3$ in the $g'$ direction.

We first focus on $T_2' = T_2$, i.e.\ we consider the map $\mathcal C^{k,\alpha}\ni c' \mapsto \mathcal P(-\zeta^{k_0}c'/2) \in \overline{\zeta\mathcal A}^{k,\alpha}$. If $c'(\zeta) = \sum_{n \in \mathbb Z} c'_n \zeta^n$ with $c'_n = \overline{c'_{-n}}$, we have 
$$\mathcal P(-\zeta^{k_0}c'/2) = -\frac{1}{2}\sum_{n< 0} c'_{n-k_0}\zeta^n.$$ Since the indices $n-k_0$ are all negative, this map is clearly onto. 

We turn now to $T'_1$. The derivative of $T'_1$ in the $h'$ direction consists of two summands: computing the first one we obtain, setting $h' = (1-\zeta)h''$ with $h''\in \mathcal A^{k,\alpha}$ and $\beta_j = j(j-1)\alpha_{j}$, and taking in account that $\overline \zeta = 1/\zeta$,
$$
\begin{array}{lll}
\frac{\zeta^{k_0}h'(\zeta)}{s(\zeta)(1-\zeta)^{d-1}} P_{z^2}(1-\zeta, 1-\overline\zeta) & = &\frac{h''(\zeta)}{s(\zeta)(1-\zeta)^{d-2}} \sum_{j=d-k_0}^{k_0} \beta_j (1-\zeta)^{j-2} \zeta^{k_0}(1-\overline \zeta)^{d-j}\\
\\
&= &\frac{h''(\zeta)}{s(\zeta)(1-\zeta)^{d-2}} \sum_{j=d-k_0}^{k_0} \beta_j (1-\zeta)^{j-2} \zeta^{k_0+j-d}(\zeta-1)^{d-j}\\
\\
& = &\frac{h''(\zeta)}{s(\zeta)}\sum_{j=d-k_0}^{k_0} (-1)^{d-j}\beta_j \zeta^{k_0+j-d}.
\end{array}
$$
Since in the sum the exponents $k_0 + j - d$ are always non-negative, this term belongs to $\mathcal A^{k,\alpha}$ for any $h''\in \mathcal A^{k,\alpha}$, hence its projection vanishes. On the other hand, computing the second summand we get
$$
\begin{array}{lll}
\frac{\overline{h'(\zeta)}}{s(\zeta)(1-\zeta)^{d-1}}  \zeta^{k_0}P_{z\overline z}(1-\zeta, 1-\overline\zeta) &=& 
-\frac{\overline{h''(\zeta)}}{\zeta s(\zeta)(1-\zeta)^{d-2}} (\zeta - 1)^{d-2} Q(\zeta)\\
\\
&= &(-1)^dC\overline{h''(\zeta)}t(\zeta)
\end{array}
$$
Since $t(\zeta)$ is a  holomorphic polynomial whose zeros are contained in the unit disc $\Delta$, 
Lemma \ref{lemfred} implies that  the map $(T'_1)_Y$ is surjective at $p_0$ in the direction of $h'$.

Finally we have to consider the derivative of $T'_3$ in the $g'$ direction. Proving its surjectivity amounts to show that for any  $v\in \mathcal C_0^{k,\alpha}$ there exists $g'\in (1-\zeta)\mathcal A^{k,\alpha}$ such that $2\real g' = v$. Thus $\real g'$ is the harmonic extension of $v/2$ to the unit disc and $\imag g'$ is uniquely determined as the harmonic conjugate of $v/2$ such that $\imag g'(1) = 0$. It follows that $g'(1) = 0$, i.e.\ $g' = (1-\zeta)g''$ for some holomorphic function $g''$ continuous up to $\overline \Delta$ (in fact $g''$ is at least in $\mathcal A^{k-1}$ since $g'\in \mathcal A^{k,\alpha}$). We want to show that $g''\in \mathcal A^{k,\alpha}$. By definition of $\mathcal C_0^{k,\alpha}$, we have $v = (1-\zeta)v'$ for some $v'\in \mathcal C_\C^{k,\alpha}$. We can write for $\zeta \in b\Delta$
\[(1-\zeta)g''(\zeta) + (1 - \overline \zeta) \overline {g''(\zeta)} = (1-\zeta) v'(\zeta)\]
and since $(1- \overline \zeta) = -\overline\zeta (1-\zeta)$ for $\zeta \in b\Delta$,
\[g''(\zeta) - \overline {\zeta g''(\zeta)} = v'(\zeta).\]
Applying the Szego projection $\mathcal P': \mathcal C_\C^{k,\alpha} \to \mathcal A^{k,\alpha}$ to both sides, we conclude that $g'' = \mathcal P'(v') \in \mathcal A^{k,\alpha}$. By the discussion above, this finishes the proof of the surjectivity of $T'_Y$ at $p_0=(\rho, c_0, h_0, g_0)$.
\end{proof}

Now, we show that the kernel of $T'_Y$ at $p_0=(\rho, c_0, h_0, g_0)$ is finite dimensional:
\begin{lemma} \label{findim}
The kernel of $T'_Y(p_0)$ has real dimension $2(k_0 +\ell_0)+3=4k_0-d+3$.
\end{lemma}
\begin{proof} We will revisit the computations performed in Lemma \ref{surj}, and start solving $(T'_2)_Y(p_0)[c',h',g']=\mathcal P \left(-\frac{\zeta^{k_0}}{2}c' \right)=0$. 
According to Lemma \ref{lemc}, $c'$ must be of the form (\ref{eqc}). 
Thus the projection of $\ker T'_Y(p_0)$ to the factor $\mathcal C^{k,\alpha}$ has real dimension $2k_0 + 1$.

Next, looking at the expression of $(T'_1)_Y(p_0)$, and taking in account that, as proved in Lemma \ref{surj}, 
$\frac{\zeta^{k_0}h'(\zeta)}{s(\zeta)(1-\zeta)^{d-1}} P_{z^2}(1-\zeta, 1-\overline\zeta) \in \mathcal A^{k,\alpha}$ for any $h'\in (1-\zeta)\mathcal A^{k,\alpha}$, we see that for any given $c'$ of the form  (\ref{eqc}), we need to solve for $h'$ the equation
\begin{equation}\label{needtosol}
\mathcal P \left( \frac{\overline{h'(\zeta)}}{s(\zeta)(1-\zeta)^{d-1}}  \zeta^{k_0}P_{z\overline z}(1-\zeta, 1-\overline\zeta) \right) = v 
\end{equation}
where
\[v = - \mathcal P \left(\frac{\zeta^{k_0}}{s(\zeta)(1-\zeta)^{d-1}} P_z(1-\zeta, 1-\overline\zeta) c'  \right ) \in \overline {\mathcal A}^{k,\alpha}.\]
Since any two solutions of (\ref{needtosol}) differ by a solution of its homogenized version (i.e.\ with $v=0$), it is of course enough to consider the latter. In view of the computations in Lemma \ref{surj}, this amounts to solving
\begin{equation}\label{needtosol2}
\mathcal P \left(t(\zeta) \overline{h''(\zeta)}\right ) = 0
\end{equation}
for $h'' \in \mathcal A^{k,\alpha}$. According to Lemma \ref{lemfred}, the kernel of 
$\mathcal P (t.)$ has complex dimension $\ell_0+1$. Thus 
the (affine) space of solutions of (\ref{needtosol}) has real 
dimension 
$2(\ell_0+1)$, and in turn that the projection of $\ker T'_Y(p_0)$ to the factor $\mathcal C^{k,
\alpha}\times (1-\zeta)\mathcal A^{k,\alpha}$ has real dimension  $2k_0 + 2\ell_0+ 3$.

Finally, considering $(T'_3)_Y$, given any $h'$ which satisfies (\ref{needtosol}) we must solve for $g'\in (1-\zeta)\mathcal A^{k,\alpha}$ the equation
\[ \frac{g'}{2} + \frac{\overline g'}{2} = P_z(1-\zeta, 1-\overline\zeta) h' +  P_{\overline z}(1-\zeta, 1-\overline\zeta) \overline h'.\]
The same proof as in Lemma \ref{surj}, however, shows that the previous equation admits a unique solution $g'$. By the previous arguments, we conclude that the kernel of $T'_Y(p_0)$ has  real dimension $2(k_0 + \ell_0) + 3$.
\end{proof}

We have proved that the derivative at the point $p_0=(\rho, c_0, h_0, g_0)$ of  
$T'_Y(p_0): Y \to (\mathcal{C}_\C^{k,\alpha})^2 \times \mathcal C_0^{k,\alpha}$ is surjective
and that its kernel  $\ker T'_Y(p_0)$ has real dimension $4k_0-d+3$. In particular, the 
Banach space $Y$ can be decomposed as a direct sum $Y= \ker T'_Y(p_0) \oplus W$, 
where $W$ is isomorphic to $(\mathcal{C}_\C^{k,\alpha})^2 \times \mathcal C_0^{k,\alpha}$.
According to the implicit function theorem, there exist open neighborhoods $V$ of $\rho$ in $X$, $U$ 
of $0$ in $\ker T'_Y(p_0)\simeq \R^{4k_0-d+3}$, and $U'$ of $(c_0,f_0)$ in $W$,   
and a map $v: V \times U \to U'$,
 such that $T'(r,t\oplus w)=0$ if and only if $w=v(r,t)$. It follows that the 
 $\mathcal C^1$ map  $\mathcal{F}:V \times U \to Y$ defined by $\mathcal{F}(r,t)=t\oplus v(r,t)$ 
satisfies the properties of Theorem  \ref{theodiscs}.

\section{Properties of $k_0$-stationary discs}

Let $S= \{\rho = 0\}$ be a model hypersurface of type $d$ 
$$\rho(z,w)  = - \real w + P(z,\overline z) =  - \real w + \sum_{j=d-k_0}^{k_0} \alpha_{j} z^j \overline{z}^{d-j}.$$
Let $k$ be an integer and let  $0<\alpha<1$.
Fix a $k_0$-stationary disc $f_0(\zeta)=(1-\zeta,g_0)$ with $c_0\equiv 1$.
For $\eta>0$, we denote by $\mathcal{S}^{r}_{0,\eta}$ the set 
$$
\mathcal{S}^{r}_{0,\eta} = \{(c,f)=(c,h,g) \in Y \ | \ (c,f) \in \mathcal{S}_{0}^r,\ \|(c,f)-(1,f_0)\|_Y<\eta \}.
$$
Let $U$, $V$, $\eta>0$, and $\mathcal F:U\times V\to \mathcal{S}^{r}_{0,\eta}$ given by Theorem \ref{theodiscs}. 
We write 
$$\mathcal{F}(r,t)=(c_{r,t},h_{r,t},g_{r,t}).$$ Define a map $\Phi: U \times V \rightarrow \R \times \C^{k_0}\simeq \R^{2k_0+1}$ by 
$$\Phi(r,t)=\left(c_{r,t,0},c_{r,t,1},\cdots,c_{r,t,k_0}\right)$$
where  $c_{r,t}(\zeta)=\sum_{n< 0} {c_{r,t,n}}\zeta^{n}+ c_{r,t,0}+ \sum_{n>0} c_{r,t,n}\zeta^n$ with $c_{r,t,n}=\overline{c_{r,t,-n}}$. 
With a slight abuse of notations, we will also write $\Phi(r,c,f)=\Phi\left(r,\mathcal F(r,.)^{-1}(c,f)\right)$; this notation will be used in the proof of Theorem \ref{theojet}. 
\begin{lemma}\label{lemsub}
Shrinking the neighborhoods $U$ and $V$ if necessary, the map $\Phi$  is a  submersion. 
\end{lemma}
\begin{proof}
According to the proof of Lemma \ref{findim} and Lemma \ref{lemc}, the derivative $\frac{\partial \Phi}{\partial t}(\rho,0)$ is of rank $2k_0+1$. 
Thus for $r$  in a neighborhood of $\rho$ and $t$ sufficiently small, the rank of $\frac{\partial \Phi}{\partial t}(r,t)$ is also $2k_0+1$. 
This proves that for a fixed $\rho$, the map
$\Phi(\rho,.): V \rightarrow \R^{2k_0+1}$ is  a submersion.
\end{proof}

We define a real $2$-dimensional submanifold $\Gamma$ of $\mathbb R \times \C^{k_0}$ in the following way:  
$$
\begin{array}{l}
\Gamma = \{(c_0,\ldots, c_{k_0}) \in \mathbb R \times \C^{k_0} | \ c(b) = (\overline b\overline \zeta + 1 + b\zeta)^{k_0}

=
\sum_{n=-k_0}^{-1}\overline{c_n}\zeta^n 
+ c_0 + \sum_{n=1}^{k_0}c_n\zeta^n, \\
\\  
  \hspace{12cm}  b\in \C, |b|<1/2 \},
\end{array}
$$ 
An easy computation shows that 
$$
\left\{
\begin{array}{lll}
c_0(b) & = & 1+ O(|b|^2), \\
\\
c_1(b) & = & k_0 b + O(|b|^2), \\
\\
c_n(b) & = & O(|b|^2) \mbox{ for } 2\leq n\leq k_0. 
\end{array}
\right.
$$
It follows that the tangent space of $\Gamma$ at the point $c_0=1$ is generated by the vectors $(0,1,0,\ldots,0)$ and $(0,i,0,\ldots,0)$.
Using Lemma \ref{lemsub}, the set $\Phi^{-1}(\Gamma)$ is a smooth submanifold of $U \times V$, and, for $r \in U$, 
 the submanifold
$$\mathcal M^r_{\eta} = 
\mathcal S^r_{0,\eta} \cap \mathcal F(\Phi^{-1}(\Gamma))$$ 
has real dimension $2\ell_0 + 4=2k_0-d+4=4k_0-d+3-(2k_0-1)$. 
Henceforth, we will restrict our attention to the submanifold $\mathcal M^r_{\eta}$ and we will study 
the properties of discs in that submanifold.    
\begin{remark}
Notice that the point 
$p_0=(\rho,c_0,h_0,g_0)$ belongs to $\mathcal M^\rho_{\eta}$. More in general, for $b\in \C$ such that $|b|<1/2$, define 
$p(b) = (\rho, c(b), h(b), g(b))$ with $c(b)=(\overline b\overline \zeta + 1 + b\zeta)^{k_0}$,  
$h(b) = \frac{1-\zeta}{1- a(b)\zeta}$ where $a(b) = \frac{-1+\sqrt{1-4|b|^2}}{2b}$ and $g(b)$ is uniquely determined by $h(b)$. Then $p(b)$ belongs to $\mathcal M^{\rho}_{\eta}$ 
for all $b\in \mathbb C$ with $|b|<1/2$. We will use this special family of $k_0$-stationary discs alongside with Lemma \ref{findim} to compute the tangent 
space of $\mathcal M^{\rho}_{\eta}$ at $p_0$.
\end{remark}
Define $\pi: Y\to (1-\zeta)\mathcal A^{k,\alpha} \times (1-\zeta)\mathcal A^{k,\alpha}$ to be the projection on the $(h,g)$ factor and $\pi': Y\to (1-\zeta)
\mathcal A^{k,\alpha}$ to be the projection on the $h$ factor.

\begin{lemma} \label{piprime}
 The restrictions of $\pi'$ and  of $\pi$ to $T_{p_0}\mathcal M^\rho_{\eta}$ are injective.
\end{lemma}
\begin{proof}
We first note that, as follows from the proofs of Lemmas \ref{surj} and \ref{findim}), for any $(c',h',g')\in T_{p_0}\mathcal M^\rho_{\eta}$ we have that $g'$ is uniquely determined by $h'$. It immediately follows that $\pi$ is injective if and only if $\pi'$ is injective, and in the rest of the proof we can ignore the factor relative to $g'$.

We start by computing the tangent space of the real $2$-dimensional submanifold of 
$\mathcal M^\rho_{\eta}$ given by the parametrization $p(b)$ defined above, at $p_0 = p(0)$. Developing $c(b)$ and $h(b)$ up to first order in $b$ and taking in account that $a(b) = -\overline b + O(|b|^2)$, we get
$$
\left\{
\begin{array}{lll} 
c(b) &=& \overline b\overline \zeta + 1 + b\zeta + O(|b|^2),\\
\\
h(b) &=&  (1-\zeta)(1 - \overline b \zeta) + O(|b|^2).
\end{array}
\right.
$$
It follows that the (projection to the $(c,h)$ factor of the) tangent space of the submanifold parametrized by $p(b)$ at $p_0$ is  generated over $\mathbb R$ by 
$(\overline \zeta + \zeta, -(1-\zeta)\zeta)$ and $(-i\overline \zeta + i\zeta, i(1-\zeta)\zeta)$. In view of the proof of Lemma \ref{findim},  the elements of (the projection to the $(c,h)$ space of) 
$T_{p_0}\mathcal M^\rho_{\eta}$ are the vectors given by
\begin{equation}\label{vectby}
\left ( 2t_1\real \zeta - 2t_2 \imag \zeta, (1-\zeta)\cdot ((-t_1+ it_2)\zeta + h''(\zeta))\  \right )
\end{equation}
for $t_1,t_2\in \mathbb R$ and where $h''$ is a solution of (\ref{needtosol2}) for $v=0$.

Define now the linear subspace $\mathcal H''\subset \mathcal A^{k,\alpha}$  as the span over $\C$ of $\zeta$ and of homogeneous solutions 
of (\ref{needtosol2}). 
Moreover, set $
\mathcal H' = 
(1-\zeta)\mathcal H'' \subset (1-\zeta)\mathcal A^{k,\alpha}$. By (\ref{vectby}), it follows that $\mathcal H'$ is the image of the restriction of $\pi'$ to 
$T_{p_0}\mathcal M^\rho_{\eta}$, and $\pi'$ is injective if and only if the function $\zeta$ is not a solution of (\ref{needtosol2}). Recal that $t(\zeta) = \prod_{i_0 + 1\leq j \leq 2k_0 - d} (q_j - \zeta)$. We have 
$$\mathcal P(t(\zeta)\zeta)=\frac{\prod_{i_0 + 1\leq j \leq 2k_0 - d} q_j}{\zeta}.$$  
Since none of the roots of $t$ is equal to zero we have $\mathcal P(t(\zeta)\zeta)\neq 0$.

\end{proof}
 In particular, notice that $\dim_\R \mathcal H' = 2\ell_0 + 4$.
 
Define now, for any $n\in \mathbb N$, $n\leq k$, the {\it $n$-jet map at $1$} as the complex linear map $\mathfrak j_n: (1-\zeta)\mathcal A^{k,\alpha}\to \mathbb C^n$ given by
\[\mathfrak j_n(h) = (h'(1), h''(1),\ldots, h^{(n)}(1)) \]
for any $h\in (1-\zeta)\mathcal A^{k,\alpha}$.

\begin{prop}\label{propprop}
Let $S= \{\rho = 0\}$ be a model hypersurface of finite type $d$. Fix a $k_0$-stationary disc $f_0(\zeta)=(1-\zeta,g_0(\zeta))$ with 
$c_0\equiv1$, attached to 
$S$.  
Then
\begin{enumerate}[i)]
\item There exists a sufficiently small $\theta \in \R$ such that the derivative of the evaluation map $\varphi:\pi'(\mathcal M^\rho_{\eta})\rightarrow \C^2$ defined by $\varphi(h)=(h(0),g(0))$, at the point $h_\theta(\theta)=(1-\zeta)e^{i\theta}$, is surjective.   
\item  Shrinking the neighborhood $U$ of $\rho$, and $\eta >0$, given in Theorem \ref{theodiscs},  
for  $r \in U$, for $\ell_0=k_0-d/2$,
 the restriction of  $\mathfrak j_{\ell_0 + 2}$ to $\pi'(\mathcal M^r_{\eta})$ is a diffeomorphism onto its image.
\end{enumerate}
\end{prop}
In particular, point $i)$ of Proposition \ref{propprop} implies that the set $\{(h(0),g(0)) \in \C^2 | (h,g) \in \pi(\mathcal M^\rho_{\eta})\} \subseteq 
\{(h(0),g(0)) \in \C^2 | (c,h,g) \in \mathcal S^\rho_0\}$ contains an open set.  

\subsection{Proof of $i)$ of Proposition \ref{propprop}}
Define the  composed map 
$$\psi: \pi'(\mathcal M^{\rho}_{\eta}) \to \pi(\mathcal M^{\rho}_{\eta}) \to \C^2$$
defined by $\psi: h \mapsto (h,g) \mapsto (h(0),g(0))$. 
Recall that $(h,g) \in \mathcal{S}^{\rho}_0$ satisfy for  $\zeta \in b\Delta$
$$\real g(\zeta)= P(h(\zeta),\overline{ h(\zeta)}).$$
Thus, from $g(1)=0$ and from classical facts on Cauchy transform (see Lemma 3 in \cite{ro}), we have 
 $$g(0)=\frac{1}{i\pi}\int_{b\Delta}\frac{\real g(\zeta)}{1-\zeta}\frac{d\zeta}{\zeta}=\frac{1}{i\pi}\int_{b\Delta}\frac{P(h(\zeta),\overline{ h(\zeta)})}{1-\zeta}\frac{d\zeta}{\zeta}.$$ 
We consider the subset $\mathcal M':=\left\{h \in \pi'(\mathcal M^{\rho}_{\eta}) | \ h(\zeta)=\frac{1-\zeta}{1-a\zeta}v,  
a \in \Delta, v \in \C^*\right\}$. Notice that for such a disc $h$,  the tangent disc $h' \in T_h\mathcal M'$ can be written 
$$h'(\zeta)=\frac{(1-\zeta)\zeta}{(1-a\zeta)^2}va'+\frac{1-\zeta}{1-a\zeta}v'$$
where $(a',v') \in \C^2.$ Consider the disc $h_{\theta} \in \mathcal M'$ given by $h_{\theta}(\zeta)=(1-\zeta)e^{i\theta} \in \mathcal M'$ for $\theta$ 
sufficiently small. 
The derivative of $\psi$ at $h_{\theta} \in \mathcal M'$ is given by 
\begin{eqnarray*}
d_{h_{\theta}}\psi(h')&=&(h'(0),d_{h_{\theta}} g(0) (h'))\\
\\
&=&\left(h'(0),
\frac{1}{i\pi}\int_{b\Delta}\frac{P_z (h_{\theta}(\zeta),\overline{ h_{\theta}(\zeta)})h'(\zeta)}{1-\zeta}\frac{d\zeta}{\zeta}+
\frac{1}{i\pi}\int_{b\Delta}\frac{P_{\overline z}(h_{\theta}(\zeta),\overline{ h_{\theta}(\zeta)})\overline{h'(\zeta)}}{1-\zeta}\frac{d\zeta}{\zeta}
\right).
\end{eqnarray*} 
where $h' \in T_{h_{\theta}}\mathcal M'$. 
 Now we fix $(Z,W) \in \C^2$ and we solve $d_{h_{\theta}}\psi(h')=(Z,W)$. The first component gives $v'=Z$ and the second component leads to 
 \begin{eqnarray*}
 i\pi W&=&\int_{b\Delta}P_z \left(h_{\theta}(\zeta),\overline{ h_{\theta}(\zeta)}\right)Z\frac{d\zeta}{\zeta}+\int_{b\Delta}P_z \left(h_{\theta}(\zeta),\overline{h_{\theta}(\zeta)}\right)e^{i\theta}a'd\zeta\\
 \\
&&   -\int_{b\Delta}P_{\overline z}\left(h_{\theta}(\zeta),\overline{ h_{\theta}(\zeta)}\right)\overline{Z}\frac{d\zeta}{\zeta^2}-\int_{b\Delta}P_{\overline z}\left(h_{\theta}(\zeta),\overline{ h_{\theta}(\zeta)}\right)e^{-i\theta}\overline{a'}\frac{d\zeta}{\zeta^3}.  
 \end{eqnarray*} 
 Set  $I_1(\theta)=\int_{b\Delta}P_z \left(h_{\theta}(\zeta),\overline{h_{\theta}(\zeta)}\right)d\zeta$ and 
 $I_2(\theta)=-\int_{b\Delta}P_{\overline z}\left(h_{\theta}(\zeta),\overline{ h_{\theta}(\zeta)}\right)\frac{d\zeta}{\zeta^3}$. It follows that $d_{h_{\theta}}\psi(h')$ is surjective if and only if 
 $|I_1(\theta)|^2\neq |I_2(\theta)|^2$.  
 A straightforward computation leads to 
 $$I_1(\theta)=-\sum_{j=d-k_0}^{k_0}{d-1\choose d-1-j}j\alpha_je^{i(2j-d-1)\theta}.$$
 $$I_2(\theta)=\sum_{j=d-k_0}^{k_0-2}{d-1\choose d-3-j}(d-j)\alpha_je^{i(2j-d+1)\theta}$$    
if $d-3<k_0$, and  
 $$I_2(\theta)=\sum_{j=d-k_0}^{k_0}{d-1\choose d-3-j}(d-j)\alpha_je^{i(2j-d+1)\theta}.$$   
 otherwise. The highest degree term of the trigonometric polynomial $|I_1(\theta)|^2$ is given by 
 $$I_1(\theta)'=k_0(d-k_0)\alpha_{k_0}\overline{\alpha_{d-k_0}}{d-1\choose k_0-1}{d-1\choose d-1-k_0}j\alpha_je^{i(4k_0-2d)\theta}.$$
 In case  $k_0\leq d-3$, the highest degree term of the  trigonometric polynomial  $|I_2(\theta)|^2$ is given by 
 $$I_2(\theta)'=k_0(d-k_0)\alpha_{k_0}\overline{\alpha_{d-k_0}}{d-1\choose k_0-3}{d-1\choose d-3-k_0}j\alpha_je^{i(4k_0-2d)\theta}.$$
 Since ${d-1\choose k_0-1}{d-1\choose d-1-k_0}<{d-1\choose k_0-3}{d-1\choose d-3-k_0}$, there exist a sufficiently small $\theta$ such that  $|I_1(\theta)|^2\neq |I_2(\theta)|^2$. 
 Finally if $d-3<k_0$, the degrees of the trigonometric polynomial of $|I_2(\theta)|^2$ and  $|I_1(\theta)|^2$ are different 
 and therefore  $|I_1(\theta)|^2\neq |I_2(\theta)|^2$ for some sufficiently small $\theta$.
 \qed

\subsection{Proof of $ii)$ of Proposition \ref{propprop}} In order to prove the second part of Proposition \ref{propprop}. 
it is sufficient to show that the restriction of $\mathfrak j_{\ell_0 + 2}$ to 
$\pi'(T_{p_0}\mathcal M^{\rho}_{\eta})$ is injective.

We first compute explicitly the tangent space of $\mathcal M^\rho_{\eta}$ at $p_0$. 
Recall that the elements of (the projection to the $(c,h)$ space of) 
$T_{p_0}\mathcal M^\rho_{\eta}$ are given by
\begin{equation*}
\left ( 2t_1\real \zeta - 2t_2 \imag \zeta, (1-\zeta)\cdot ((-t_1+ it_2)\zeta + h''(\zeta))\  \right )
\end{equation*}
for $t_1,t_2\in \mathbb R$ and where $h''$ is a solution of (\ref{needtosol2}) for $v=0$.
We need to describe explicitly the solutions $h''$ of (\ref{needtosol2}) for $v=0$.
If $t(\zeta)\equiv 1$ put $\ell_0=0$, otherwise write $t(\zeta) = \sum_{i=0}^{\ell_0} t_i \zeta^i$ where $\ell_0$ is 
defined in Lemma \ref{surj} and $t_i \in \mathbb C$. We need to find 
$h''(\zeta) = \sum_{n\geq 0} h''_n \zeta^n$ such that $t(\zeta)\overline{h''(\zeta)}
\in \mathcal A^{k,\alpha}$. This translates into the following recursion:
\begin{equation}\label{traninto}
t_{\ell_0} \overline h''_{n+\ell_0} + \ldots + t_1 \overline h''_{n+1} + t_0 \overline h''_{n} = 0
\end{equation}
for all $n\geq 1$. 
Therefore, we must examine the solutions of (\ref{traninto}). Let $r_1, \ldots, r_{\ell_1}$ be the \emph{distinct} roots of the polynomial $t(\zeta)$, with, respectively, multiplicity $m_1,\ldots, m_{\ell_1}$, so that $\sum_{j=1}^{\ell_1} m_j = \ell_0$. Moreover, put $r_0 = 0$, $m_0=1$.
 The general solution of (\ref{traninto}) is then given by $h''_n = \sum_{j=0}^{\ell_1} Q_j(n) \overline r_j^n$, $n\geq 0$, where with a slight abuse of notation we set $(r_0)^0 = 1$ and where each $Q_j$ is a polynomial of degree at most $m_j-1$ with complex coefficients. Put $R_j(\zeta) = \sum_{n\geq 0} \overline r_j^n \zeta^n = \frac{1}{1 - \overline r_j \zeta}$. It is easy to see that, for any fixed polynomial $Q_j$ of degree $m_j - 1$, there exist $q_{j,0}, \ldots, q_{j,m_j-1}\in \mathbb C$ such that $\sum_{n\geq 0} Q_j(n)\overline r_j^n \zeta^n= \sum_{i=0}^{m_j-1} q_{j,i} \zeta^i R_j^{(i)}(\zeta)$. Hence for the general solution $h''(\zeta)$ of the homogeneous version of (\ref{needtosol2}) we get
\[h''(\zeta) =\sum_{j=0}^{\ell_1} \sum_{n\geq 0} Q_j(n)\overline r_j^n \zeta^n= \sum_{j=0}^{\ell_1} \sum_{i=0}^{m_j-1} q_{j,i} \zeta^i R_j^{(i)}(\zeta) = \sum_{j=0}^{\ell_1} \sum_{i=0}^{m_j-1} i! \cdot q_{j,i} \frac{\overline r_j^i\zeta^i }{(1-\overline r_j\zeta)^{i+1}}\]
for $q_{j,i}\in \mathbb C$. Note that, for all $1\leq j \leq \ell_1$, $0\leq i \leq m_j - 1$
\[\frac{\overline r_j^i\zeta^i }{(1-\overline r_j\zeta)^{i+1}} = \frac{(1 - (1 - \overline r_j\zeta))^i }{(1-\overline r_j\zeta)^{i+1}} = \sum_{\kappa=0}^i (-1)^\kappa \binom{i}{\kappa} \frac{1}{(1-\overline r_j\zeta)^{\kappa+1}}, \]
hence we can also write the general solution as
\begin{equation}\label{genersolut}
h''(\zeta) = s_{0,0} + \sum_{j=1}^{\ell_1} \sum_{i=0}^{m_j-1} s_{j,i} \frac{1}{(1-\overline r_j\zeta)^{i+1}}
\end{equation}
for $s_{j,i}\in \mathbb C$.  It follows that the elements of (the projection to the $(c,h)$ space of) $T_{p_0}\mathcal M^\rho_{\eta}$ are the vectors given by
\begin{equation}\label{vectby2}
\left ( 2t_1\real \zeta - 2t_2 \imag \zeta, (1-\zeta)\cdot ((-t_1+ it_2)\zeta + h''(\zeta))\  \right )
\end{equation}
for $t_1,t_2\in \mathbb R$ and $h''(\zeta)$ as in (\ref{genersolut}).

Let $\mathcal H' = \pi'(T_{p_0}\mathcal M^\rho_{\eta})$ be the $(2\ell_0 + 4)$-dimensional space introduced in Lemma \ref{piprime}, and set
\[u_0 = 1, \ u_1 = \zeta, \  \left \{u_{j,i} = \frac{1}{(1-r_j\zeta)^{i+1}} \right \}_{\substack{1\leq j\leq j_1, \\ 0\leq i\leq m_j - 1}}. \]
By definition, a basis of $\mathcal H'$ is then given by $(1-\zeta)u_0$, $(1-\zeta)u_1$ and $(1-\zeta)u_{j,i}$. By the Leibniz rule, for every $v(\zeta) = (1-\zeta)u(\zeta)$ and $n\geq 1$ we have $v^{(n)}(\zeta) = (1-\zeta)u^{(n)} - nu^{(n-1)}(\zeta)$, so that  $v^{(n)}(1) = - nu^{(n-1)}(1)$. Since
\[\frac{d^n}{ d\zeta^n} \left( \frac{1}{(1-r_j\zeta)^{i+1}} \right) = \frac{(i+n)!}{i!} \cdot \frac{r_j^n}{(1-r_j\zeta)^{i+n+1}},\] 
it follows that, for every $n>2$, the expression of $\mathfrak j_n$ in the given basis is 

\[  (-1) \cdot     \begin{pmatrix}
1		&   1  	& \frac{1}{1-r_1}    			  	 & \frac{1}{(1-r_1)^2} 			    &\cdots  & \frac{1}{(1-r_{\ell_1})^{m_{\ell_1}}} 			\\
0  		&   2         	& 2\frac{r_1}{(1-r_1)^2}	  		 &  2\frac{2r_1}{(1-r_1)^3}	    		    &\cdots  & 2\frac{m_{\ell_1}r_{\ell_1}}{(1-r_{\ell_1})^{m_{\ell_1} +1 }}   	\\
0		&    0	 	& 3\frac{2r_1^2}{(1-r_1)^3}		 & 3\frac{6r_1^2}{(1-r_1)^4} 		    &\cdots  & 3\frac{m_{\ell_1}(m_{\ell_1}+1)r_{\ell_1}^2}{(1-r_{\ell_1})^{m_{\ell_1} +2 }}  	     \\
\vdots	&  \vdots 	&   \vdots		     		  	  	 & 		\vdots				    &\ddots  & \vdots			 \\
0		&    0	 	& n \frac{(n-1)! r_1^{n-1}}{(1-r_1)^{n}}&  n \frac{n! r_1^{n-1}}{(1-r_1)^{n+1}}  &\cdots  & n \frac{(m_{\ell_1}+n-2)!r_{\ell_1}^{n-1}}{(m_{\ell_1}-1)!(1-r_{\ell_1})^{m_{\ell_1} +n-1 }}		 \\
		       \end{pmatrix}.
\]
Taking $n=\ell_0+2$, the one above is a square matrix of size $\ell_0+2$. Put now $\chi_j = \frac{r_j}{1-r_j}$. After a suitable sequence of scalar multiplications of the row/columns, we can transform the previous matrix into the following one:
\[			\begin{pmatrix}
1		&   1  	& 1    			& 1			    		        &\cdots  & 1			\\
0  		&   1         	& \chi_1	  		&  2\chi_1	    		    	        &\cdots  & m_{\ell_1} \chi_{\ell_1}  	\\
0		&    0	 	& \chi_1^2		& 3\chi_1^2		    	        &\cdots  & \binom{m_{\ell_1}+1}{2} \chi_{\ell_1}^2 	     \\
\vdots	&  \vdots 	&   \vdots		     	& 		\vdots		        &\ddots  & \vdots			 \\
0		&    0	 	&  \chi_1^{\ell_0+1}  &   (\ell_0+2)\chi_1^{\ell_0+1}  &\cdots  &  \binom{m_{\ell_1}+ \ell_0}{\ell_0 +1}  \chi_{\ell_1}^{\ell_0+1}		 \\
		        \end{pmatrix}.
\]
Next, performing a sequence of subtractions of columns and keeping in account the properties of the binomial coefficients, we get
\[			\begin{pmatrix}
1	 &   1      & 1    		         & -2			    		   &\cdots& 1 		                      &-2 				      &\cdots& (-1)^{m_{\ell_1}+1}(m_{\ell_1} + 1)			\\
0  	 &   1      & \chi_1	         &  -\chi_1	    		    	   &\cdots& \chi_{\ell_1}                  &- \chi_{\ell_1} 		      &\cdots& (-1)^{m_{\ell_1}+1}\chi_{\ell_1}  	\\
0	 &    0     & \chi_1^2	         &   0				    	   &\cdots& \chi_{\ell_1}^2              & 0  				      &\cdots& 0	     \\
0	 &   0      & \chi_1^3	         &  \chi_1^3			   &\cdots& \chi_{\ell_1}^3              &\chi_{\ell_1}^3 		      &\cdots& 0						\\
\vdots&\vdots&   \vdots	         & 		\vdots		   &\ddots& \vdots                             & \vdots                                  &\ddots& \vdots			 \\
0	 &    0    &\chi_1^{\ell_0+1}&(\ell_0-1)\chi_1^{\ell_0+1}&\cdots&\chi_{\ell_1}^{\ell_0+1}&(\ell_0-1)\chi_1^{\ell_0+1}&\cdots&\binom{\ell_0 - 1}{\ell_0 - m_{\ell_1} }  \chi_{\ell_1}^{\ell_0+1}  \\
		        \end{pmatrix}.
\]
Finally, a further sequence of column multiplications gives

\[			\begin{pmatrix}
1	 &   1      &\frac{1}{\chi_1^2}&\frac{-2}{\chi_1^3}		   &\cdots&\frac{1}{\chi_{\ell_1}^2}&\frac{-2}{\chi_{\ell_1}^3}  &\cdots& \frac{(-1)^{m_{\ell_1}+1}(m_{\ell_1} + 1)(m_{\ell_1} - 1)!}{\chi_{\ell_1}^{m_{\ell_1}+1}}			\\
0  	 &   1      &\frac{1}{\chi_1}	  &\frac{-1}{\chi_1^2}	    	   &\cdots& \frac{1}{\chi_{\ell_1}}    &\frac{-1}{\chi_{\ell_1}^2}  &\cdots& \frac{(-1)^{m_{\ell_1}+1}(m_{\ell_1} - 1)!}{\chi_{\ell_1}^{m_{\ell_1}}}  	\\
0	 &    0     & 1       	                    &   0				    	   &\cdots& 1              			    & 0  				     &\cdots& 0	     \\
0	 &   0      & \chi_1	           &  1			   	  	   &\cdots& \chi_{\ell_1}                   & 1 		           	     &\cdots& 0						\\
\vdots&\vdots&   \vdots	           & 		\vdots		   &\ddots& \vdots                             & \vdots                                 &\ddots& \vdots			 \\
0	 &    0    &\chi_1^{\ell_0-1}  &(\ell_0-1)\chi_1^{\ell_0-2}&\cdots&\chi_{\ell_1}^{\ell_0-1} &(\ell_0-1)\chi_1^{\ell_0-2}&\cdots&\frac{(\ell_0 - 1)!}{(\ell_0  - m_{\ell_1})! }  \chi_{\ell_1}^{\ell_0  - m_{\ell_1}}  \\
		        \end{pmatrix}.
\]

The $\ell_0\times \ell_0$ minor obtained by erasing the first two rows and columns of the previous matrix is a confluent Vandermonde matrix: it is a well-known fact that its determinant does not vanish. It follows that the determinant of the whole matrix is in turn non-vanishing, hence $\mathfrak j_{\ell_0 + 2}$ is injective as claimed.

\qed

\section{Finite jet determination}

\subsection{Dilation}

Let $\mathcal U$ be a neighborhood of $0$ in $\mathbb C^2$, and let $r:\mathcal U\to \mathbb R$ be a real valued function of class $\mathcal C^{d+k+3}$ that can be written as (\ref{eqdef3}). 
Moreover, let $H: \mathcal U\to \C^2$ be a local biholomorphism,  tangent to the identity up to order $\ell > d$, fixing $0$. We can write $H=(H_1,H_2)$, where
$$
\left\{
\begin{array}{lll}
H_1(z,w) &=& z + \sum_{j+l=\ell+1}z^jw^l H_1^{jl}(z,w),\\ 
\\
H_2(z,w) &= & w + \sum_{j+l=\ell+1}z^jw^l H_2^{jl}(z,w),
\end{array}
\right.
$$
for suitable $H_1^{jl}, H_2^{jl}\in \mathcal O(\mathcal U)$. Furthermore,
\[ r(z,w) = \rho(z,w) + \sum_{i+j=d+1} (z^i \overline z^j) \cdot r_{ij0}(z)+
\sum_{l=1}^{d-1}\sum_{i+j=d-l} z^i \overline z^j (\imag w)^l \cdot r_{ijl}(z,\imag w)+\theta_1(\imag w) \]
for suitable $r_{ij0} \in  \mathcal C^{k+3}_\C$, 
$r_{ijl}\in \mathcal  C^{k+3}_\C$, $\theta_1 \in \mathcal C^{k+3}$ 
such that   $\theta_1(\imag w)=O(|\imag w|^2)$.

For $t> 0$, denote by $\phi_t$ the linear map 
$$\phi_t(z,w) = (tz,t^dw),$$ 
and let $r_t = \frac{1}{t^d}r \circ \phi_t$ and  $H_t=\phi_t^{-1}\circ H \circ \phi_t$ be the direct images of,
respectively, $r$ and $H$, 
under $\phi_t$. For any $\delta>0$,  if $t>0$ is small enough, we have
 $\phi_t\left(\overline {\delta\Delta}^2\right)\subset \mathcal U$, hence $H_t$ is defined on 
 $\overline {\delta\Delta}^2$ and $r_t\in X$.
We first show that if $t$ is small enough then $r_t$ is a small perturbation of $\rho$: 
\begin{lemma}\label{rt} 
We have $\|r_t - \rho\|_X \to 0$ as $t\to 0$.
\end{lemma}
\begin{proof}
Indeed,
\[ r_t(z,w) = \rho(z,w) +  \sum_{i+j=d+1} (z^i \overline z^j) \cdot r^t_{ij0}(z)+
\sum_{l=1}^{d-1}\sum_{i+j=d-l} z^i \overline z^j (\imag w)^l \cdot r^t_{ijl}(z,\imag w)+\theta^t_1(\imag w)\] 
with
$$
\left\{
\begin{array}{lll}
r^t_{ij0}(z) &= &t \ r_{ij0}(tz), \\
\\
r^t_{ijl}(z,w) &= &t^{l(d-1)} r_{ijl}(tz,t^d\imag w) \ \mbox{ for } l\geq 1,\\
\\
\theta^t_1(\imag w) &= & O(|t|^d).
\end{array}
\right.
$$
It is clear that $\|r^t_{ijl}\|_{\mathcal C^{k+3}(\overline {\delta\Delta}^2)}\to 0$ and 
$\|\theta^t_1\|_{\mathcal C^{k+3}([\delta,\delta])}\to 0$ as $t\to 0$ , and the claim of the lemma follows from the definition of the norm in the space $X$.
\end{proof}
We will employ now the notation $Z = (1-\zeta)\mathcal A^{k,\alpha}\times (1-\zeta)\mathcal A^{k,\alpha}$ and the norm 
$$\|(h,g)\|_Z=\|h\|_{(1-\zeta)\mathcal A^{k,\alpha}}+
\|g\|_{(1-\zeta)\mathcal A^{k,\alpha}}.$$
\begin{lemma}\label{Ht}
For all $f = (h,g) \in Z$ and $t>0$ small enough we have $H_t\circ f\in Z$ and 
$$\|H_t\circ f - f\|_Z \leq tK\|f\|_Z^{\ell+2}$$ 
for some constant $K>0$.
\end{lemma} 
\begin{proof}
Indeed we have $H_t(z,w) = (H_1^t(z,w),H_2^t(z,w))$, with
$$
\left\{
\begin{array}{lll}
H_1^t(z,w) = z + \sum_{j+l=\ell+1}t^{j+dl-1}z^jw^l H_1^{jl}(tz,t^dw),\\
\\
H_2^t(z,w) = w + \sum_{j+l=\ell+1}t^{j+dl-d}z^jw^l H_2^{jl}(tz,t^dw).
\end{array}
\right.
$$
Posing $h= (1-\zeta)\tilde{h}, g = (1-\zeta)\tilde{g}$, with $\tilde{h},\tilde{g}\in \mathcal A^{k,\alpha}$, we get
$$
\left\{
\begin{array}{lll}
H_1^t(h,g) &= &h + (1-\zeta)^{\ell+1}\sum_{j+l=\ell+1}t^{j+dl-1}\tilde{h}^j\tilde{g}^l H_1^{jl}(t(1-\zeta)\tilde{h},t^d(1-\zeta)\tilde{g}),\\
\\
H_2^t(h,g) &=& g + (1-\zeta)^{\ell+1}\sum_{j+l=\ell+1}t^{j+dl-d}\tilde{h}^j\tilde{g}^l H_2^{jl}(t(1-\zeta)\tilde{h},t^d(1-\zeta)\tilde{g}).
\end{array}
\right.
$$
It follows that 
$$
\left\{
\begin{array}{lll}
\|H_1^t(h,g) - h\|_{(1-\zeta)\mathcal A^{k,\alpha}} &=& \left \|(1-\zeta)^\ell \sum_{j+l=\ell+1}t^{j+dl-1}\tilde{h}^j\tilde{g}^l H_1^{jl}(t(1-\zeta)\tilde{h},t^d(1-\zeta)\tilde{g}) \right \|_{\mathcal C_{\C}^{k,\alpha}}, \\
\\
\|H_2^t(h,g) - g\|_{(1-\zeta)\mathcal A^{k,\alpha}} &=& \left \|(1-\zeta)^\ell \sum_{j+l=\ell+1}t^{j+dl-d}\tilde{h}^j\tilde{g}^l H_2^{jl}(t(1-\zeta)\tilde{h},t^d(1-\zeta)\tilde{g})\right \|_{\mathcal C_{\C}^{k,\alpha}}.
\end{array}
\right.
$$
The estimate of the lemma then follows from the previous expressions, using the facts that $j+dl-d\geq 1$ whenever $j+l=\ell+1$, that the norm of $(1-\zeta)$ in $\mathcal A_{k,\alpha}$ is finite, and that the $\mathcal C^{k,\alpha}$ norm of a composition is estimated by a constant times the $\mathcal C^{k,\alpha}$ norms of the maps which are composed.
\end{proof}

\subsection{Finite jet determination of biholomorphic maps}
Let $\mathcal U\subset \mathbb C^2$ and 
$r\in \mathcal C^{d+(\ell_0+2)+ 3}(\mathcal U)= \mathcal C^{d/2+k_0+ 5}(\mathcal U)$ be as in the previous subsection. Define $M=\{r=0\}$, and let $H$ be a local biholomorphism such that $H(0)=0$, $H(M)\subset M$ and $H$ is tangent to the identity up to order $\ell_0 + 2 = k_0 -d/2 + 2$.
\begin{theorem}\label{theojet}
 Under the assumptions above, $H$ coincides with the identity.
\end{theorem}
\begin{proof}
According to the first point of Proposition \ref{propprop}, for every sufficiently small $\epsilon >0$ there 
exists a $k_0$-stationary disc $f \in Z$ attached to the model hypersurface $S=\{\rho=0\}$ 
with coefficient $c\in \mathcal C^{k,\alpha}$ such that $\|(c,f)-(1,f_0)\|_Y<\epsilon$ and the restriction of 
the evaluation map $\varphi$, as a map from $\mathcal M^\rho_{\eta}$ to $\C^2$, to the tangent space of $\mathcal M^\rho_{\eta}$ at $(c,f)$ is onto. 
As a 
consequence we have the following: for every sufficiently small 
$\epsilon>0$, there exists $\epsilon'>0$ such that for all $r
\in X$ with $\|r-\rho\|_X<\epsilon'$ there exists a $k_0$-stationary disc $f\in Z$ attached to $\{r=0\}$ with 
coefficient function $c\in \mathcal C^{k,\alpha}$ such that $\|(c,f)-(1,f_0)\|_Y<\epsilon$ and the 
restriction of the evaluation map $\varphi$ to the tangent space of $\mathcal M^r_{\epsilon}$ at $(c,f)$ is onto.

Fix now $\epsilon>0$ such that $\epsilon<\eta/2$, where $\eta$ is defined in Theorem \ref{theodiscs}. 
Using Lemma \ref{rt}, we can find $t>0$ such that $r_t\in V$ and $\|r_t - \rho\|<\epsilon'$, where $V$ is 
neighborhood of $\rho$ in $X$ identified in Theorem \ref{theodiscs}. We can further require that $t< 
\epsilon/2^{\ell_0+4}K\|f_0\|^{\ell_0+4}_Z$ with $K$ as in the statement of Lemma \ref{Ht}. In view of the 
previous paragraph, we can find a $k_0$-stationary disc $(c_1,f_1)\in \mathcal M^{r_t}_{\epsilon}$, 
such that the evaluation map $\varphi$ is locally surjective 
in a neighborhood $(c_1,f_1)$.

Hence, we can find a neighborhood $O'$ of $(c_1,f_1)$ in $\mathcal M^{r_t}_{\epsilon}$ such that $\varphi(O') = O$ 
is an open subset of the bidisc $(\delta \Delta)^2$. 
Choose any $q\in O$, and let $(c_q,f_q)
\in O'\subset \mathcal M^{r_t}_{\epsilon}$ be such that $f(0)=q$. Note that $\|f_q\|_Z \leq \|f_0\|_Z + \epsilon\leq 2\|f_0\|_Z$ if $\epsilon$ is 
small enough. Since $H_t$ is tangent to the identity up to order $\ell_0 + 2$, we can apply Lemma 
\ref{Ht} with $\ell=\ell_0+2$ to get 
\begin{equation}\label{eqestH}
\|H_t\circ f_q - f_q \|_Z \leq tK\|f_q\|_Z^{\ell_0+4} \leq tK 2^{\ell_0+4} \|f_0\|^{\ell_0+4}\leq \epsilon
\end{equation}
by the choice of $t$. On the other hand, by Proposition \ref{statinvar} the disc $H_t\circ f_q$ is again a 
stationary disc attached to $\{r_t=0\}$, with coefficient function $c_q$, and tied to the origin. 
Moreover, from (\ref{eqestH})
  follows that 
$$\|(c_q, H_t\circ f_q) - (c_0,f_0)\|_Y \leq 2\epsilon < \eta.$$
Thus $(c_q,H_t\circ f_q) \in \mathcal S^{r_t}_{0,\eta}$, 
and since $\Phi(r_t, c_q, H_t\circ f_q) = \Phi(r_t, c_q, 
f_q) \in \Gamma$ we deduce that $(c_q,H_t\circ f_q)\in \mathcal M^{r_t}_{\eta}$.

Since $H_t$ is tangent to the identity up to order $\ell_0 + 2$, we also have that the holomorphic disc 
$H_t\circ f_q(\zeta)$ is tangent to $f_q(\zeta)$ at $\zeta=1$ up to order $\ell_0+2$. The second point of 
Proposition \ref{propprop} then implies that $H_t\circ f_q \equiv f_q$, and computing at $\zeta=0$, we get 
in particular $H_t(q)=q$. Since this holds for any $q$ belonging to the open set $O$, we conclude that 
$H_t$, and therefore $H$, coincides with the identity map.
\end{proof}

As stressed in \cite{be-bl}, the approach based on invariants such as $k_0$-stationary discs attached to a given hypersurface $M
$, allows one to deal with one-sided biholomorphisms that extend smoothly up to $M$. 

The  situation we first consider is the case of germs of CR diffeomorphisms between two pseudoconvex hypersurfaces 
of the form (\ref{eqdef3}). Such a germ admits a local one-sided holomorphic extension \cite{ba-tr, tu1, tr} that extends smoothly 
up to the given hypersurface (Theorem 7.5.1 in \cite{ber}), and therefore, one obtain: 
\begin{theorem}\label{theocr}
Let $M,\ M'\subset \C^{2}$ be two pseudoconvex hypersurfaces  
of the form (\ref{eqdef3}) at $0\in M$. Suppose that $M$ is of class $\mathcal{C}^{d/2+k_0+5}$. 
Then the germs at $0$ of CR diffeomorphisms of class $\mathcal{C}^{k_0-d/2+2}$
 between $M$ and $M'$ are uniquely determined by their $(k_0-d/2+2)$-jet at $0$.
\end{theorem}

Since any  biholomorhisms between two smooth bounded pseudoconvex domains of finite type extends up to the boundary \cite{be-li, di-fo}, 
one obtain a boundary version of H. Cartan's uniqueness theorem: 
\begin{theorem}\label{theobd}
Let $D,\ D'\subset \C^{2}$ be two smooth bounded pseudoconvex   domains, and let $p\in bD$ be such that the local defining function of $D$ has the form (\ref{eqdef3}). 
If $H_1$ and $H_2$ are two biholomorhisms from $D$ onto $D'$ with the same $(k_0-d/2+2)$-jet at $p$, they coincide.
\end{theorem}
 
We note that the previous theorem is in the vein of some previous boundary rigidity results involving smooth hypersurfaces. 
D. Burns and  S. Krantz 
\cite{bu-kr} proved that if $D\subset \C^n$ is a smooth bounded and strongly pseudoconvex domain, then 
if $H: D \rightarrow D$ is holomorphic and satisfies $H(Z)=Z+o(|Z-p|^3)$ as $Z\rightarrow p \in bD$, 
then $H$ coincides with the identity. As shown by X. Huang \cite{hu1}, in case 
$D$ is  strongly convex, the assumption reduces to $H(Z)=Z+o(|Z-p|^2)$. Moreover, if  $D \subset \C^n$ 
is a smooth bounded convex domain of finite type, then there exists a number $\ell$, depending only 
the geometric properties of $bD$ near $p$,  such that if $H(Z)=Z+o(|Z-p|^m)$ as $Z\rightarrow p \in bD$, then  
$H$ coincides with the identity; for instance in $\C^2$, the number $\ell$ can be taken to be $5d+\varepsilon$ for some $\epsilon>0$, 
where $d$ is the type of $bD$ at $p$.   
As a matter of fact, we point out that the methods developped in \cite{bu-kr, hu1} were based on Lempert's theory \cite{le} of 
complex geodesics for the Kobayashi metric in convex bounded domains.

\vskip 0,5cm
{\small
\noindent Florian Bertrand\\
Department of Mathematics, University of Vienna, Nordbergstrasse 15, Vienna, 1090, Austria\\
{\sl E-mail address}: florian.bertrand@univie.ac.at\\
\\
Giuseppe Della Sala \\
Department of Mathematics, University of Vienna, Nordbergstrasse 15, Vienna, 1090, Austria\\
{\sl E-mail address}: 	giuseppe.dellasala@univie.ac.at\\
}

\end{document}